\documentclass[a4paper,UKenglish,cleveref, autoref, theorem-restate,authorcolumns]{lipics-v2019}

\usepackage{stmaryrd}
\usepackage{amssymb}
\usepackage{amsmath}
\usepackage{graphicx}
\usepackage{proof}
\usepackage{enumerate}
\usepackage{amsthm}
\usepackage{appendix}
\usepackage{enumitem}
\usepackage{multicol}
\usepackage{subcaption}

\usepackage{epigraph}
\usepackage{tikz}
\usetikzlibrary{arrows, topaths, calc, positioning,shapes,intersections}

\tikzstyle{node} = [circle, minimum size = 1mm, inner sep = 0mm, color=black, fill]
\tikzstyle{hyperedge} = [rectangle, minimum width = 5mm, minimum height = 5mm, draw, inner sep = 0mm, color = black]
\tikzstyle{HG} = [align = center]
\tikzset{>=stealth}

\newcommand{\downsquigarrow}{\mathbin{\rotatebox[origin=c]{-90}{$\rightsquigarrow$}}}
\newcommand{\Rev}{{\mathrm{R}}}
\newcommand{\eqdef}{\mathrel{\mathop:}=}
\mathchardef\mhyp="2D
\newcommand{\FOInt}{\mathrm{FO\mhyp Int}}

\title{Introduction to a Hypergraph Logic Unifying Different Variants of the Lambek Calculus} 

\author{Tikhon Pshenitsyn}{Lomonosov Moscow State Universuty, Moscow, Russia\and \url{https://www.researchgate.net/profile/Tikhon\_Pshenitsyn}}{ptihon@yandex.ru}{}{}

\authorrunning{T. Pshenitsyn} 

\Copyright{Tikhon Pshenitsyn}



\category{} 

\relatedversion{https://arxiv.org/abs/2010.00819} 

\supplement{}

\funding{The study was funded by RFBR, project number 20-01-00670 and by the Interdisciplinary Scientific and Educational School of Moscow University ``Brain, Cognitive Systems, Artificial Intelligence''. The author is a Scholarship holder of ``BASIS'' Foundation.}

\acknowledgements{I thank my scientific advisor Prof. Mati Pentus for revising my work and giving me valuable advice.}

\nolinenumbers 

\EventEditors{}
\EventNoEds{0}
\EventLongTitle{}
\EventShortTitle{}
\EventAcronym{}
\EventYear{}
\EventDate{}
\EventLocation{}
\EventLogo{}
\SeriesVolume{}
\ArticleNo{}

\begin{document}

\maketitle

\begin{abstract}
	In this paper \textit{hypergraph Lambek calculus} ($\mathrm{HL}$) is presented. This formalism aims to generalize the Lambek calculus ($\mathrm{L}$) to hypergraphs as hyperedge replacement grammars extend context-free grammars. In contrast to the Lambek calculus, $\mathrm{HL}$ deals with hypergraph types and sequents; its axioms and rules naturally generalize those of $\mathrm{L}$. Consequently, certain properties (e.g. the cut elimination) can be lifted from $\mathrm{L}$ to $\mathrm{HL}$. It is shown that $\mathrm{L}$ can be naturally embedded in $\mathrm{HL}$; moreover, a number of its variants ($\mathrm{LP}$, $\mathrm{NL}$, $\mathrm{NLP}$, $\mathrm{L}$ with modalities, $\mathrm{L}^\ast(\mathbf{1})$, $\mathrm{L}^{\mathrm{R}}$) can also be embedded in $\mathrm{HL}$ via different graph constructions. We also establish a connection between $\mathrm{HL}$ and Datalog with embedded implications. It is proved that the parsing problem for $\mathrm{HL}$ is NP-complete.
\end{abstract}
\section{Introduction}\label{sec_intr}
	In the work of Joachim Lambek \cite{Lambek58}, the logical approach is introduced, which is based on algebraic structures and which is used to describe natural languages. Here we consider its standard variant called the Lambek calculus (in the Gentzen style); below it is denoted as $\mathrm{L}$. Besides its mathematical significance (it can be considered as a fragment of linear or of intuitionistic logic), the Lambek calculus forms the basis for Lambek grammars, which serve to describe formal languages. Such grammars assign logical types to symbols and accept a string if a sequent composed of these types is derivable in $\mathrm{L}$. Lambek grammars are opposed to context-free grammars, which generate strings using productions.
	
	Since 1958 until nowadays the Lambek calculus has been significantly improved, its different extensions have been presented in a number of works. Some extensions that are of interest in this paper are the following: $\mathrm{L}^\ast_{\mathbf{1}}$ (Lambek calculus with the unit); $\mathrm{L}$ with modalities, which is presented in  \cite{Moortgat96}; $\mathrm{L}$ with the reversal operation; $\mathrm{L}$ with the permutation rule $\mathrm{LP}$. One of the objectives of such extensions is to capture and naturally describe difficult linguistic phenomena (see e.g. \cite{Moortgat96}). 
	
	Many fundamental properties of the Lambek calculus and of its extensions have been discovered. Pentus proved that the derivability problem in $\mathrm{L}$ is NP-complete (see \cite{Pentus06}); L-models and R-models were introduced, completeness was established in \cite{Pentus95}. Regarding Lambek grammars, it is proved in \cite{Pentus93} that the class of languages generated by Lambek grammars equals the class of context-free languages without the empty word. 
	
	The second field of research that should be mentioned in this work is the theory of graph grammars. Generalizing context-free grammars (CFGs), graph grammars produce graph languages using rewriting rules. An overview of graph grammars is given in the handbook \cite{Rozenberg97}; a wide variety of mechanisms generating graphs is presented there. We focus on a particular approach called \emph{hyperedge replacement grammar} (HRG in short) since it is very close to context-free grammars in terms of definitions and structural properties. Hyperedge replacement grammars generate hypergraphs by means of productions: a production allows one to replace an edge of a hypergraph by another hypergraph. Hyperedge replacement grammars (HRGs) have a number of properties in common with CFGs, such as the pumping lemma, the Parikh theorem, the Greibach normal form etc. An overview of HRGs can be found in \cite{Drewes97}; in this study we stick to definitions from that book.
	
	From the practical point of view, HRGs became popular in an NLP context; namely, it turns out that HRGs can serve to describe \emph{abstract meaning representation} (AMR), which is useful for machine translation: a sentence is converted into its AMR and then --- from the AMR to a sentence of some other language. There is a number of works related to this topic, e.g. Bauer and Rambow \cite{Bauer16}, Gilroy et al. \cite{Gilroy17}, Jones et al. \cite{Jones12}, Peng et al. \cite{Peng15}. However, the area of applications is not limited by AMR. Another field where HRGs can be used is programming, e.g., program verification. See related works by Jansen et al. \cite{Jansen11}, and by Mazanek and Minas \cite{Mazanek08}.
	
	Being impressed by many similarities between HRGs and CFGs, we were curious whether it is possible to generalize type-logical grammars to hypergraphs. We started with basic categorial grammars and introduced hypergraph basic categorial grammars (see \cite{Pshenitsyn20}); we showed their duality with HRGs and also a number of similarities with basic categorial grammars. Our goal now is to construct a similar generalization of the Lambek calculus, which would deal with hypergraphs. We wish to preserve fundamental features of the Lambek calculus, such as the cut elimination, L- and R-models as well as to extend Lambek grammars in such a way that there will be duality between them and hyperedge replacement grammars.
	
	In this work we present a solution meeting the above requirements. It is called \emph{the hypergraph Lambek calculus} ($\mathrm{HL}$). In Section \ref{sec_prelim} we introduce preliminary definitions. In Section \ref{sec_HL} we present axioms and rules of the new calculus. Such features of $\mathrm{HL}$ as the cut elimination and complexity are discussed in Section \ref{sec_prop}. There we also show that the syntax of $\mathrm{HL}$ can be modelled in Datalog with embedded implications (using the language of the first-order intuitionistic logic), thus unveiling connections between them; correctness of such an embedding is proved. In Section \ref{sec_embed}, it is shown that the Lambek calculus and its variants mentioned above can be naturally embedded in HL; consequently, HL may be considered as an extensive source of ``well-behaved'' modifications of $\mathrm{L}$.
	
	The work is rather introductory; our main goal is to convince the reader that $\mathrm{HL}$ is an appropriate generalization of the Lambek calculus. Unforunately, we do not discuss here the issue of extending Lambek grammars; however, this is possible, and in that field a number of nontrivial and even unexpected results are established. They are presented in our work \cite{Pshenitsyn20Preprint}.
	
\section{Preliminaries}\label{sec_prelim}
\subsection{Lambek calculus}\label{ssec_Lambek}
In this section we provide basic definitions regarding the Lambek calculus; concepts behind them form the basis for the idea of the hypergraph Lambek calculus.

Let us fix a countable set $Pr=\{p_i\}_{i=1}^\infty$ of \emph{primitive types}. 
\begin{definition}
	The set $Tp(\mathrm{L})$ of types in the Lambek calculus is the least set such that:
	\begin{itemize}
		\item $Pr\subseteq Tp(\mathrm{L})$;
		\item If $A,B\in Tp(\mathrm{L})$ are types, then $(B\backslash A), (A/B), (A\cdot B)$ belong to $Tp(\mathrm{L})$ (brackets are often omitted).
	\end{itemize}
\end{definition}
Everywhere below the set of types of a calculus $\mathrm{C}$ is denoted as $Tp(\mathrm{C})$.
\begin{definition}
	A \emph{sequent} is of the form $T_1,\dots,T_n\to T$ where $T_i, T$ are types ($n>0$). $T_1,\dots,T_n$ is called an antecedent, and $T$ is called a succedent.
\end{definition}
When talking about the Lambek calculus, small letters $p, q, \dots$ and strings composed of them (e.g. $np, cp$) range over primitive types. Capital letters $A,B,\dots$ range over types. Capital Greek letters $\Gamma,\Delta,\dots$ range over finite (possibly empty) sequences of types. Sequents thus can be represented as $\Gamma\to A$, where $\Gamma$ is nonempty. 

The Lambek calculus $\mathrm{L}$ is a logical system with one axiom and six inference rules ($\Pi,\Psi$ being nonempty):

$$\infer[]{p\to p}{}
$$
$$
\infer[(\backslash\to)]{\Gamma, \Pi, A \backslash B, \Delta \to C}{\Pi \to A & \Gamma, B, \Delta \to C}
\qquad
\infer[(\to\backslash)]{\Pi \to A \backslash B}{A, \Pi \to B}
\qquad
\infer[(\cdot\to)]{\Gamma, A \cdot B, \Delta \to C}{\Gamma, A, B, \Delta \to C}
$$
$$
\infer[(/\to)]{\Gamma, B / A, \Pi, \Delta \to C}{\Pi \to A & \Gamma, B, \Delta \to C}
\qquad
\infer[(\to/)]{\Pi \to B / A}{\Pi, A \to B}
\qquad
\infer[(\to\cdot)]{\Pi, \Psi \to A \cdot B}{\Pi \to A & \Psi \to B}
$$
A sequent $\Gamma\to A$ is derivable ($\mathrm{L}\vdash \Gamma\to A$) if it can be obtained from axioms applying rules. A corresponding sequence of rule applications is called a derivation.

\subsection{Hypergraphs}\label{sec_hypergraph}
According to our expectations, the hypergraph Lambek calculus must deal with sequents, axioms and rules directly generalizing those of the Lambek calculus; besides, it has to be strongly connected to hyperedge replacement grammars (HRGs) in the sense of underlying mechanisms and properties. To satisfy the latter, below we introduce definitions related to hypergraphs according to the well-known handbook chapter \cite{Drewes97} on HRGs.

$\mathbb{N}$ includes $0$. $\Sigma^*$ is the set of all strings over the alphabet $\Sigma$ including the empty string $\Lambda$. $\Sigma^\circledast$ is the set of all strings consisting of distinct symbols. The length $|w|$ of the word $w$ is the number of symbols in $w$. The set of all symbols contained in a word $w$ is denoted by $[w]$. If $f:\Sigma\to\Delta$ is a function from one set to another, then it is naturally extended to a function $f:\Sigma^*\to\Delta^*$ ($f(\sigma_1\dots\sigma_k)=f(\sigma_1)\dots f(\sigma_k)$).

Let $C$ be some fixed set of labels for whom the function $type: C\to \mathbb{N}$ is considered. 
\begin{definition}\label{def_hypergraph}
	\emph{A hypergraph $G$ over $C$} is a tuple $G=\langle V, E, att, lab, ext \rangle$ where $V$ is a set of \emph{nodes}, $E$ is a set of \emph{hyperedges}, $att: E\to V^\circledast$ assigns a string (i.e. an ordered set) of \emph{attachment} nodes to each edge, $lab: E \to C$ labels each edge by some element of $C$ in such a way that $type(lab(e))=|att(e)|$ whenever $e\in E$, and $ext\in V^\circledast$ is a string of \emph{external} nodes. 
	
	Components of a hypergraph $G$ are denoted by $V_G, E_G, att_G, lab_G, ext_G$ resp.
\end{definition}
In the remainder of the paper, hypergraphs are simply called graphs, and hyperedges are called edges. Usual graphs (with type 2 hyperedges) are called 2-graphs. The set of all graphs with labels from $C$ is denoted by $\mathcal{H}(C)$. Graphs are usually named by letters $G$ and $H$.

In drawings of graphs, black dots correspond to nodes, labeled squares correspond to edges, $att$ is represented by numbered lines, and external nodes are depicted by numbers in brackets. If an edge has exactly two attachment nodes, it can be denoted by an arrow (which goes from the first attachment node to the second one). 
\begin{definition}\label{type}
	$type_G$ (or $type$, if $G$ is clear) returns the number of nodes attached to an edge in a graph $G$: $type_G(e)\eqdef|att_G(e)|$.
	If $G$ is a graph, then $type(G)\eqdef |ext_G|$.
\end{definition}
\begin{example}
	The following picture represents a graph $G$:
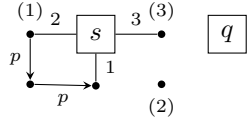
	\begin{center}
		{\tikz[baseline=.1ex]{
				\node[] (R) {};
				\node[node,above right=2mm and 0mm of R,label=above:{\scriptsize $(1)$}] (N1) {};
				\node[node,below=5.5mm of N1] (N2) {};
				\node[hyperedge,right=5.5mm of N1] (E) {$s$};
				\node[node,below=3.7mm of E] (N3) {};
				\node[node,right=5.5mm of E,label=above:{\scriptsize $(3)$}] (N4) {};
				\node[node,below=5.5mm of N4,label=below:{\scriptsize $(2)$}] (N5) {};
				\node[hyperedge,right=5.5mm of N4] (E2) {$q$};

				\draw[>=stealth,->,black] (N1) -- node[left] {\scriptsize $p$} (N2);
				\draw[>=stealth,->,black] (N2) -- node[below] {\scriptsize $p$} (N3);
				\draw[-,black] (N3) -- node[right] {\scriptsize 1} (E);
				\draw[-,black] (N1) -- node[above] {\scriptsize 2} (E);
				\draw[-,black] (N4) -- node[above] {\scriptsize 3} (E);
		}}
	\end{center}
	Here $type(p)=2, type(q)=0, type(s)=3$; $type(G)=3$.
\end{example}
\begin{definition}
	A sub-hypergraph (or just subgraph) $H$ of a graph $G$ is a hypergraph such that $V_H\subseteq V_G$, $E_H\subseteq E_G$, and for all $e\in E_H$ $att_H(e)=att_G(e)$, $lab_H(e)=lab_G(e)$.
\end{definition}
\begin{definition}
	If $H=\langle \{v_i\}_{i=1}^n,\{e_0\},att,lab,v_1\dots v_n\rangle$, $att(e_0)=v_1\dots v_n$ and $lab(e_0)=a$, then $H$ is called \emph{a handle}. In this work we denote it by $a^\bullet$.
\end{definition}
\begin{definition}
	\emph{An isomorphism} between graphs $G$ and $H$ is a pair of bijective functions $\mathcal{E}: E_G\to E_H$, $\mathcal{V}: V_G\to V_H$ such that $att_H\circ\mathcal{E}=\mathcal{V}\circ att_G$, $lab_G=lab_H\circ\mathcal{E}$, $\mathcal{V}(ext_G)=ext_H$. 
\end{definition}
In this work, we do not distinguish between isomorphic graphs.

Strings can be considered as graphs with the string structure. This is formalized in
\begin{definition}\label{def_str_gr}
	A string graph induced by a string $w=a_1\dots a_n$ is a graph of the form $\langle \{v_i\}_{i=0}^n,\{e_i\}_{i=1}^n,att,lab,v_0v_n \rangle$ where $att(e_i)=v_{i-1}v_i$, $lab(e_i)=a_i$. It is denoted by $w^\bullet$.
\end{definition}
We additionally introduce the following definition (not from \cite{Drewes97}):
\begin{definition}
	Let $H\in \mathcal{H}(C)$ be a graph, and let $f:E_H\to C$ be a function. Then $f(H)=\langle V_H, E_H, att_H, lab_{f(H)}, ext_H\rangle$ where $lab_{f(H)}(e)=f(e)$ for all $e$ in $E_H$. It is required that $type(lab_H(e))=type(f(e))$ for $e\in E_H$.
\end{definition}
If one wants to relabel only one edge $e_0$ within $H$ with a label $a$, then the result is denoted by $H[e_0\eqdef a]$.

\subsection{Hyperedge replacement}\label{ssec_repl}
This procedure is defined in \cite{Drewes97} and it plays a fundamental role in hyperedge replacement grammars. The replacement of an edge $e_0$ in $G$ with a graph $H$ can be done if $type(e_0)=type(H)$ as follows:
\begin{enumerate}
	\item Remove $e_0$;
	\item Insert an isomorphic copy of $H$ (namely, $H$ and $G$ have to consist of disjoint sets of nodes and edges);
	\item For each $i$, fuse the $i$-th external node of $H$ with the $i$-th attachement node of $e_0$.
\end{enumerate}
The result is denoted by $G[e_0/H]$. It is known that if several edges of a graph are replaced by other graphs, then the result does not depend on the order of replacements; moreover the result is not changed if replacements are done simultaneously. In \cite{Drewes97} this is called sequentialization and parallelization properties. The following notation is in use: if $e_1,\dots,e_k$ are distinct edges of a graph $H$ and they are simultaneously replaced by graphs $H_1,\dots,H_k$ resp. (this means that $type(H_i)=type(e_i)$), then the result is denoted $H[e_1/H_1,\dots,e_k/H_k]$.

\section{Hypergraph Lambek Calculus: Definitions}\label{sec_HL}
In this section we introduce the hypergraph Lambek calculus: we define types, sequents, axioms and rules of this formalism. It is expected that the resulting logic will be \emph{literally} a logic on graphs: while such calculi as the Lambek calculus, the first-order predicate calculus, the propositional calculus deal with string sequents or formulas, we desire a new formalism to work with objects of graph nature. Thus sequents are supposed to be composed of graphs rather than of strings, and types are assumed to label edges of graphs. Definitions presented below meet these requirements.
\subsection{Types and sequents}
We fix a countable set $Pr$ of primitive types and a function $type:Pr\to\mathbb{N}$ such that for each $n\in\mathbb{N}$ there are infinitely many $p\in Pr$ for which $type(p)=n$. Types are constructed from primitive types using division and multiplication operations. Simultaneously, the function $type$ is defined on types (apologies for the tautology): it is obligatory since we are going to label edges by types. 

Let us fix some symbol $\$$ that is not included in any of the sets considered. {\bf NB!} This symbol is allowed to label edges with different number of attachment nodes. To be consistent with Definition \ref{def_hypergraph} one can assume that there are symbols $\$_n,n\ge 0$ instead such that $type(\$_n)=n$.
\begin{definition}
	The set $Tp(\mathrm{HL})$ of types is defined inductively as the least set satisfying the following conditions:
	\begin{enumerate}
		\item $Pr\subseteq Tp(\mathrm{HL})$.
		\item Let $N$ (``numerator'') be in $Tp(\mathrm{HL})$. Let $D$ (``denominator'') be a graph such that exactly one of its edges (call it $e_0$) is labeled by $\$$, and the other edges (possibly, there are none of them) are labeled by elements of $Tp(\mathrm{HL})$; let also $type(N)=type(D)$. Then $T=(N\div D)$ also belongs to $Tp(\mathrm{HL})$, and $type(T)\eqdef type_D(e_0)$.
		\item Let $M$ be a graph such that all its edges are labeled by types from $Tp(\mathrm{HL})$ (possibly, there are no edges at all). Then $T=\times(M)$ belongs to $Tp(\mathrm{HL})$, and $type(T)\eqdef type(M)$.
	\end{enumerate}
\end{definition}
In types with division, $D$ is usually drawn as a graph in brackets, so instead of $(N\div D)$ (a formal notation) a graphical notation $N\div (D)$ is in use. Sometimes brackets are omitted.
\begin{example}\label{ex_types}
	The following structures are types:
	\begin{itemize}
		\item $A_1=q\div\left(\mbox{
			{\tikz[baseline=1ex]{
					\node[] (V) {};
					\node[hyperedge,above=-2mm of V] (E1) {$s$};
					\node[node,right=5mm of E1] (N1) {};
					\node[node,right=10mm of N1] (N2) {};
					\node[hyperedge,right=5mm of N2] (E3) {$r$};
					\draw[-,black] (E1) -- node[above] {\scriptsize 1} (N1);
					\draw[->,black] (N1) -- node[above] {$\$$} (N2);
					\draw[-,black] (N2) -- node[above] {\scriptsize 1} (E3);
			}}
		}\right)$;
		\item $A_2=t\div\left(\mbox{
			{\tikz[baseline=1ex]{
					\node (V) {};
					\node[node,above=0.3mm of V, label=left:{\scriptsize $(1)$}] (N1) {};
					\node[hyperedge, right=5mm of N1] (E1) {$r$};
					\node[node, right=8mm of E1,label=left:{\scriptsize $(2)$}] (N2) {};
					\node[node,right=10mm of N2] (N3) {};
					\node[hyperedge,right=5mm of N3] (E3) {$s$};
					\draw[-,black] (N1) -- node[above] {\scriptsize 1} (E1);
					\draw[->,black] (N3) -- node[above] {$\$$} (N2);
					\draw[-,black] (N3) -- node[above] {\scriptsize 1} (E3);
			}}
		}\right)$;
		\item $A_3=q\div\left(\mbox{
			{\tikz[baseline=1ex]{
					\node[] (V) {};
					\node[node,above=0mm of V] (N1) {};
					\node[hyperedge, right=5mm of N1] (E1) {$\$$};
					\node[node, above right= 2mm and 9mm of E1] (N2) {};
					\node[node,below right=2mm and 9mm of E1] (N3) {};
					\draw[-,black] (N1) -- node[above] {\scriptsize 3} (E1);
					\draw[-,black] (E1) -- node[above] {\scriptsize 1} (N2);
					\draw[-,black] (E1) -- node[below] {\scriptsize 2} (N3);
					\draw[->,black] (N3) -- node[right] {$t$} (N2);
			}}
		}\right)$;
		\item $A_4=\times\left(\mbox{
			{\tikz[baseline=.1ex]{
					\node[node,label=left:{\scriptsize $(1)$}] (N1) {};
					\node[node,right=10mm of N1,label=right:{\scriptsize $(2)$}] (N2) {};
					\draw[->,black] (N1) to[bend left=50] node[above] {$p$} (N2);
					\draw[->,black] (N1) to[bend right=50] node[below] {$A_1$} (N2);
			}}
		}\right)$.
	\end{itemize}
	Here $type(p)=2, type(q)=0, type(r)=type(s)=1, type(t)=2$; $type(A_1)=type(A_2)=2,type(A_3)=3,type(A_4)=2$.
\end{example}
\begin{definition}
	\emph{A graph sequent} is a structure of the form $H\to A$, where $A\in Tp(\mathrm{HL})$ is a type, $H\in\mathcal{H}(Tp(\mathrm{HL}))$ is a graph labeled by types and $type(H)=type(A)$. $H$ is called the antecedent of the sequent, and $A$ is called the succedent of the sequent.
\end{definition}
Let $\mathcal{T}$ be a subset of $Tp(\mathrm{HL})$. We say that $H\to A$ is over $\mathcal{T}$ if $G\in\mathcal{H}(\mathcal{T})$ and $A\in\mathcal{T}$.
\begin{example}\label{ex_sequent}
	The following structure is a graph sequent:
	\\
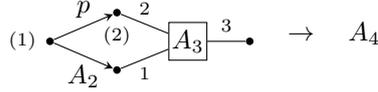
	$$\mbox{	
		{\tikz[baseline=.1ex]{
				\node[node,label=left:{\scriptsize $(1)$}] (N1) {};
				\node[node,above right=3mm and 8mm of N1,label=below:{\scriptsize $(2)$}] (N2) {};
				\node[node,below right=3mm and 8mm of N1] (N3) {};
				\node[hyperedge,right=15mm of N1] (E) {$A_3$};
				\node[node,right=5mm of E] (N4) {};
				\draw[->,black] (N1) -- node[above] {$p$} (N2);
				\draw[->,black] (N1) -- node[below] {$A_2$} (N3);
				\draw[-,black] (N3) -- node[below] {\scriptsize 1} (E);
				\draw[-,black] (N2) -- node[above] {\scriptsize 2} (E);
				\draw[-,black] (E) -- node[above] {\scriptsize 3} (N4);
	}}}\quad\to\quad A_4$$
	\\
	Here $A_i, i=2,3,4$ are from Example \ref{ex_types}.
\end{example}
\begin{remark}
	Below we provide a general intuition that should explain main principles of $\div$ and $\times$:
	\begin{itemize}
		\item Look at the context-free string production $S\to NP \mbox{ \textit{sleeps}}$. It says that a structure $S$ (sentence) may be obtained by composing a structure $NP$ (noun phrase, e.g. \textit{Tim}) with the unit \textit{sleeps} in such order. It can be rewritten as follows: $\mbox{\textit{sleeps}}\triangleright S\div(NP\;\$)$. This means that \textit{sleeps} is such a unit that if one places it instead of \$ and a structure of the type $NP$ instead of $NP$ within $(NP\;\$)$, then he obtains a structure of the type $S$. In general, a type $N\div D$ represents such units that if one places a unit on the \$-place in $D$ and fills other places of $D$ with units of corresponding types, then he obtains a structure $N$.
		\item In $\mathrm{L}$, the product like $NP\cdot VP$ represents a resource that stores both types $NP$ and $VP$. Linguistically, if $VP$ denotes verb phrases, then $NP\cdot VP$ stores, e.g., \textit{Tim sleeps}, \textit{Peter loves Helen} etc. Generally, $\times(M)$ is a type that ``freezes'' structures of different types connected to each other w.r.t. $M$ in a single new structure.
		\item A sequent $H\to A$ is understood as the following statement: ``each unit of the type $\times(H)$ is also of the type $A$''. 
	\end{itemize}
\end{remark}

\subsection{Axiom and rules}\label{sec_axioms_rules}
The hypergraph Lambek calculus (denoted $\mathrm{HL}$) we introduce here is a logical system that defines what graph sequents are derivable (=provable). $\mathrm{HL}$ includes one axiom and four rules, which are introduced below. Each rule is illustrated by examples exploiting string graphs.

The only \textbf{axiom} is the following: $p^\bullet\to p,\quad p\in Pr$.
\subsubsection{Rule $(\div\to)$.}\label{subsec_div_to}
Let $N\div D$ be a type and let $E_D=\{d_0,d_1,\dots,d_k\}$ where $lab(d_0)=\$$. Let $H\to A$ be a graph sequent and let $e\in E_H$ be labeled by $N$. Let finally $H_1,\dots,H_k$ be graphs labeled by types. Then the rule $(\div\to)$ is the following:
$$
\infer[(\div\to)]{H[e/D][d_0\eqdef N\div D][d_1/H_1,\dots,d_k/H_k]\to A}{H\to A & H_1\to lab(d_1) &\dots & H_k\to lab(d_k)}
$$
This rule explains how a type with division appears in an antecedent: we replace an edge $e$ by $D$, put a label $N\div D$ instead of \$ and replace the remaining labels of $D$ by corresponding antecedents. 
\begin{example}\label{ex_div_to}
	Consider the following rule application with $T_i$ being some types and with $T$ being equal to $q\div(T_2\$T_3)^\bullet$ (recall that $w^\bullet$ here denotes a string graph induced by $w$):
	$$
	\infer[(\div\to)]{
		(prs\,T\,tu)^\bullet\to T_1
	}{
		(pq)^\bullet\to T_1 & (rs)^\bullet\to T_2 & (tu)^\bullet\to T_3 
	}
	$$
\end{example}
\subsubsection{Rule $(\to\div)$.}
Let $F\to N\div D$ be a graph sequent; let $e_0\in E_D$ be labeled by \$. Then
	$$
	\infer[(\to\div)]{F\to N\div D}{D[e_0/F]\to N}
	$$
Formally speaking, this rule is improper since it is formulated from bottom to top. It is understood, however, as follows: if there are such graphs $D,F$ and such a type $N$ that in a sequent $H\to N$ the graph $H$ equals $D[F/e_0]$ and $H\to N$ is derivable, then $F\to N\div D$ is also derivable.
\begin{example}\label{ex_to_div}
	Consider the following rule application where $T$ equals $\times((pqr)^\bullet)$ (here we draw string graphs instead of writing $w^\bullet$ to visualize the rule application):
	$$
	\infer[(\to\div)]{
		\mbox{	
			{\tikz[baseline=.1ex]{
					\node[node,label=left:{\scriptsize $(1)$}] (N1) {};
					\node[node,right=8mm of N1] (N2) {};
					\node[node,right=8mm of N2,label=right:{\scriptsize $(2)$}] (N3) {};
					\draw[->,black] (N1) -- node[above] {$p$} (N2);
					\draw[->,black] (N2) -- node[above] {$q$} (N3);
		}}}\to T\div\left(\mbox{	
		{\tikz[baseline=.1ex]{
		\node[node,label=left:{\scriptsize $(1)$}] (N1) {};
		\node[node,right=8mm of N1] (N2) {};
		\node[node,right=8mm of N2,label=right:{\scriptsize $(2)$}] (N3) {};
		\draw[->,black] (N1) -- node[above] {$\$$} (N2);
		\draw[->,black] (N2) -- node[above] {$r$} (N3);
	}}}\right)
	}{
		\mbox{	
			{\tikz[baseline=.1ex]{
					\node[node,label=left:{\scriptsize $(1)$}] (N1) {};
					\node[node,right=8mm of N1] (N2) {};
					\node[node,right=8mm of N2] (N3) {};
					\node[node,right=8mm of N3,label=right:{\scriptsize $(2)$}] (N4) {};
					\draw[->,black] (N1) -- node[above] {$p$} (N2);
					\draw[->,black] (N2) -- node[above] {$q$} (N3);
					\draw[->,black] (N3) -- node[above] {$r$} (N4);
		}}}\to T}
	$$
\end{example}
\subsubsection{Rule $(\times\to)$.}
Let $G\to A$ be a graph sequent and let $e\in E_G$ be labeled by $\times(F)$. Then
$$
\infer[(\times\to)]{G\to A}{G[e/F]\to A}
$$
This rule again is formulated from bottom to top. Intuitively speaking, there is a subgraph of an antecedent in a premise, and it is ``compressed'' into a single $\times(F)$-labeled edge.
\begin{example}\label{ex_times_to}
	Consider the following rule application where $U$ equals $\times((pqrs)^\bullet)$:
	$$
	\infer[(\times\to)]{
		\mbox{	
			{\tikz[baseline=.1ex]{
					\node[node,label=left:{\scriptsize $(1)$}] (N1) {};
					\node[node,right=8mm of N1] (N2) {};
					\node[node,right=17.3mm of N2] (N4) {};
					\node[node,right=8mm of N4,label=right:{\scriptsize $(2)$}] (N5) {};
					\draw[->,black] (N1) -- node[above] {$p$} (N2);
					\draw[->,black] (N2) -- node[above] {$\times((qr)^\bullet)$} (N4);
					\draw[->,black] (N4) -- node[above] {$s$} (N5);	
		}}}\to U
	}{
		\mbox{	
			{\tikz[baseline=.1ex]{
					\node[node,label=left:{\scriptsize $(1)$}] (N1) {};
					\node[node,right=8mm of N1] (N2) {};
					\node[node,right=8mm of N2] (N3) {};
					\node[node,right=8mm of N3] (N4) {};
					\node[node,right=8mm of N4,label=right:{\scriptsize $(2)$}] (N5) {};
					\draw[->,black] (N1) -- node[above] {$p$} (N2);
					\draw[->,black] (N2) -- node[above] {$q$} (N3);
					\draw[->,black] (N3) -- node[above] {$r$} (N4);
					\draw[->,black] (N4) -- node[above] {$s$} (N5);	
		}}}\to U}
	$$
\end{example}
\subsubsection{Rule $(\to\times)$.}
Let $\times(M)$ be a type and let $E_M=\{m_1,\dots,m_l\}$. Let $H_1,\dots,H_l$ be graphs. Then
$$
\infer[(\to\times)]{M[m_1/H_1,\dots,m_l/H_l]\to\times(M)}{H_1\to lab(m_1) & \dots & H_l\to lab(m_l)}
$$
This rule is quite intuitive: several sequents can be combined into a single one via some graph structure $M$.
\begin{example}\label{ex_to_times}
	Consider the following rule application with $T_i$ being some types:
	$$
	\infer[(\to\times)]{
		(pqrstu)^\bullet\to \times((T_1T_2T_3)^\bullet)
	}{
		(pq)^\bullet\to T_1 & (rs)^\bullet\to T_2 & (tu)^\bullet\to T_3 
	}
	$$
\end{example}
\begin{remark}
	If a graph $M$ in the rule $(\to\times)$ does not have edges, then there are zero premises in this rule ($l=0$); hence $M\to\times(M)$ is derivable, and this sequent can be considered as an axiom.
\end{remark}
\begin{definition}
	A graph sequent $H\to A$ is \emph{derivable in $\mathrm{HL}$} ($\mathrm{HL}\vdash H\to A$) if it can be obtained from axioms using rules of $\mathrm{HL}$. A corresponding sequence of rule applications is called \emph{a derivation} and its representation as a tree is called \emph{a derivation tree}.
\end{definition}
\begin{example}\label{ex_derivation}
	The sequent from Example \ref{ex_sequent} is derivable in $\mathrm{HL}$. Here is its derivation:
	$$
	\infer[(\to\times)]
	{
		\mbox{	
			{\tikz[baseline=.1ex]{
					\node[node,label=left:{\scriptsize $(1)$}] (N1) {};
					\node[node,above right=3mm and 8mm of N1,label=below:{\scriptsize $(2)$}] (N2) {};
					\node[node,below right=3mm and 8mm of N1] (N3) {};
					\node[hyperedge,right=15mm of N1] (E) {$A_3$};
					\node[node,right=5mm of E] (N4) {};
					\draw[->,black] (N1) -- node[above] {$p$} (N2);
					\draw[->,black] (N1) -- node[below] {$A_2$} (N3);
					\draw[-,black] (N3) -- node[below] {\scriptsize 1} (E);
					\draw[-,black] (N2) -- node[above] {\scriptsize 2} (E);
					\draw[-,black] (E) -- node[above] {\scriptsize 3} (N4);
		}}}\quad\to\quad \times\left(\mbox{
	{\tikz[baseline=.1ex]{
	\node (V) {};
	\node[node,above=-1.05mm of V,label=left:{\scriptsize $(1)$}] (N1) {};
	\node[node,right=10mm of N1,label=right:{\scriptsize $(2)$}] (N2) {};
	\draw[->,black] (N1) to[bend left=50] node[above] {$p$} (N2);
	\draw[->,black] (N1) to[bend right=50] node[below] {$A_1$} (N2);
}}
}\right)
	}{
		\infer[(\to\div)]
		{
			\mbox{	
				{\tikz[baseline=.1ex]{
						\node (V) {};
						\node[node,above=3.2mm of V,label=right:{\scriptsize $(1)$}] (N1) {};
						\node[node,below =7mm of N1] (N2) {};
						\node[hyperedge,right=5mm of N2] (E3) {$A_3$};
						\node[node,above=5mm of E3,label=right:{\scriptsize $(2)$}] (N3) {};
						\node[node,right=5mm of E3] (N4) {};
						\draw[->,black] (N1) -- node[left] {$A_2$} (N2);
						\draw[-,black] (N2) -- node[below] {\scriptsize 1} (E3);
						\draw[-,black] (E3) -- node[left] {\scriptsize 2} (N3);
						\draw[-,black] (E3) -- node[below] {\scriptsize 3} (N4);
			}}}\quad\to\; q\div\left(\mbox{
		{\tikz[baseline=1ex]{
		\node[] (V) {};
		\node[hyperedge,above=-2mm of V] (E1) {$s$};
		\node[node,right=5mm of E1] (N1) {};
		\node[node,right=10mm of N1] (N2) {};
		\node[hyperedge,right=5mm of N2] (E3) {$r$};
		\draw[-,black] (E1) -- node[above] {\scriptsize 1} (N1);
		\draw[->,black] (N1) -- node[above] {$\$$} (N2);
		\draw[-,black] (N2) -- node[above] {\scriptsize 1} (E3);
	}}
}\right)
		}
		{
			\infer[(\div\to)]{
				\mbox{	
					{\tikz[baseline=.1ex]{
							\node[node] (N1) {};
							\node[hyperedge,left=4mm of N1] (E1) {$s$};
							\node[node,right=9mm of N1] (N2) {};
							\node[hyperedge,right=4mm of N2] (E2) {$A_3$};
							\node[node,above=4mm of E2] (N3) {};
							\node[node,right=4mm of E2] (N4) {};
							\node[hyperedge,right=4mm of N4] (E3) {$r$};
							\draw[-,black] (E1) -- node[above] {\scriptsize 1} (N1);
							\draw[->,black] (N1) -- node[above] {$A_2$} (N2);
							\draw[-,black] (N2) -- node[above] {\scriptsize 1} (E2);
							\draw[-,black] (E2) -- node[left] {\scriptsize 3} (N3);
							\draw[-,black] (E2) -- node[above]{\scriptsize 2} (N4);
							\draw[-,black] (N4) -- node[above] {\scriptsize 1} (E3);
				}}}\quad\to\; q
			}
			{
				\infer[(\div\to)]
				{
					\mbox{
						{\tikz[baseline=1ex]{
								\node[] (V) {};
								\node[node,above=0mm of V] (N1) {};
								\node[hyperedge, right=5mm of N1] (E1) {$A_3$};
								\node[node, above right= 1mm and 9mm of E1] (N2) {};
								\node[node,below right=1mm and 9mm of E1] (N3) {};
								\draw[-,black] (N1) -- node[above] {\scriptsize 3} (E1);
								\draw[-,black] (E1) -- node[above] {\scriptsize 1} (N2);
								\draw[-,black] (E1) -- node[below] {\scriptsize 2} (N3);
								\draw[->,black] (N3) -- node[right] {$t$} (N2);
						}}
					}\;\to\; q
				}
				{
					\infer{q^\bullet\to q}{}
					&
					\infer{t^\bullet\to t}{}
				}
				&
				\infer{r^\bullet\to r}{}
				&
				\infer{s^\bullet\to s}{}
			}
		}
		&\infer{p^\bullet\to p}{}
	}
	$$
\end{example}
The definition of the hypergraph Lambek calculus is complete. Observe that Examples \ref{ex_to_div} and \ref{ex_times_to} are related to derivations in the string Lambek calculus; this hepls one to unerstand how rules of $\mathrm{HL}$ are designed in comparison with those of $\mathrm{L}$.
\begin{remark}\label{rem_loops}
	We forbid cases where some external nodes of a graph or some attachment nodes of an edge coincide; that is, we forbid loops (in a general sense). This is done following \cite{Drewes97} (to obtain more similarities). Nevertheless our definitions can be easily extended to the cases where loops are allowed. In order to do this it suffices to replace $V^\circledast$ by $V^\ast$ in Definition \ref{def_hypergraph}. However, this leads us to questionable consequences. Consider, for instance the following two derivations with a type $L_0=\times\bigg(
	\mbox{{\tikz[baseline=.1ex]{
				\node[node] (N) {};
				\draw[>=stealth,->,black] (N) to [out=-30,in=30,looseness=30] node[right] {$p$} (N);
				\node[above=0mm of N] {\scriptsize $(1)$};
				\node[below=0mm of N] {\scriptsize $(2)$};
	}}}\bigg)$ and some type $T$:
	$$
	\infer[(\times\to)]{
		\mbox{{\tikz[baseline=.1ex]{
					\node[node] (N1) {};
					\node[node,right=5mm of N1] (N) {};
					\draw[>=stealth,->,black] (N) to [out=-30,in=30,looseness=30] node[right] {$L_0$} (N);
					\draw[>=stealth,->,black] (N1) -- node[above] {$q$} (N);
					\node[right=10mm of N] {$\to T$};
		}}}
	}{
		\mbox{{\tikz[baseline=.1ex]{
					\node[node] (N1) {};
					\node[node,right=5mm of N1] (N) {};
					\draw[>=stealth,->,black] (N) to [out=-30,in=30,looseness=30] node[right] {$p$} (N);
					\draw[>=stealth,->,black] (N1) -- node[above] {$q$} (N);
					\node[right=10mm of N] {$\to T$};
		}}}
	}
	\qquad
	\infer[(\times\to)]{
		\mbox{{\tikz[baseline=.1ex]{
					\node[node] (N1) {};
					\node[node,right=5mm of N1] (N) {};
					\node[node,right=5mm of N] (N3) {};
					\draw[>=stealth,->,black] (N) -- node[above] {$L_0$} (N3);
					\draw[>=stealth,->,black] (N1) -- node[above] {$q$} (N);
					\node[right=10mm of N] {$\to T$};
		}}}
	}{
		\mbox{{\tikz[baseline=.1ex]{
					\node[node] (N1) {};
					\node[node,right=5mm of N1] (N) {};
					\draw[>=stealth,->,black] (N) to [out=-30,in=30,looseness=30] node[right] {$p$} (N);
					\draw[>=stealth,->,black] (N1) -- node[above] {$q$} (N);
					\node[right=10mm of N] {$\to T$};
		}}}
	}
	$$
	In the case when we forbid coincidences within external or attachment nodes, the application of $(\times\to)$ is completely defined by the antecedent of a premise $H= G[F/e]$ and by the active part $F$. Moreover, one can reformulate rules $(\times\to)$ and $(\to\div)$ in a top-to-bottom way (in order to do this it suffices to define a procedure opposite to replacement). Here, however, the same sequent in the premise and its part participating in the rule can yield different conclusions; therefore, the rule $(\times\to)$ cannot be reformulated from top to bottom (it would require to introduce several options for the sequent in the conclusion). This is one of the key reasons why we decided to reject the idea of allowing coincidences of external or attachment nodes. In the remainder of the work we stick to definitions given in Sections \ref{sec_hypergraph} and \ref{sec_axioms_rules}, and return to the issue of this remark only once in Section \ref{sec_embed_l1}. However, we do not underestimate positive outcomes of considering loops since they can be useful in modelling different variants of the Lambek calculus.
\end{remark}
\section{Embedding of the Lambek calculus and of its variants in $\mathrm{HL}$}\label{sec_embed}
As expected, $\mathrm{HL}$ naturally generalizes $\mathrm{L}$. In this section we show how to embed the Lambek calculus into the hypergraph Lambek calculus considering strings as string graphs. Besides, surprisingly $\mathrm{HL}$ can model several extensions of $\mathrm{L}$, which is discussed below.
\subsection{Embedding of $\mathrm{L}$}\label{sec_embed_lambek}
Types of the Lambek calculus are embedded in $\mathrm{HL}$ by means of a function $tr:Tp(\mathrm{L})\to Tp(\mathrm{HL})$ presented below:
\begin{multicols}{2}
\begin{itemize}
	\item $tr(p)\eqdef p, \quad p\in Pr, type(p)=2$;
	\item $tr(A/B)\eqdef tr(A)\div(\$\,tr(B))^\bullet$;
	\item $tr(B\backslash A)\eqdef tr(A)\div(tr(B)\,\$)^\bullet$;
	\item $tr(A\cdot B)\eqdef \times((tr(A)tr(B))^\bullet)$.
\end{itemize}
\end{multicols}
\begin{example}
	The type $r\backslash (p\cdot q)$ is translated into the type
	$$
	\times\left(\mbox{	
		{\tikz[baseline=.1ex]{
				\node[node,label=left:{\scriptsize $(1)$}] (N1) {};
				\node[node,right=8mm of N1] (N2) {};
				\node[node,right=8mm of N2,label=right:{\scriptsize $(2)$}] (N3) {};
				\draw[->,black] (N1) -- node[above] {$p$} (N2);
				\draw[->,black] (N2) -- node[above] {$q$} (N3);
	}}}\right)\div\left(\mbox{	
		{\tikz[baseline=.1ex]{
				\node[node,label=left:{\scriptsize $(1)$}] (N1) {};
				\node[node,right=8mm of N1] (N2) {};
				\node[node,right=8mm of N2,label=right:{\scriptsize $(2)$}] (N3) {};
				\draw[->,black] (N1) -- node[above] {$r$} (N2);
				\draw[->,black] (N2) -- node[above] {$\$$} (N3);
	}}}\right)
	$$
\end{example}
String sequents $\Gamma\to A$ are transformed into graph sequents as follows: $tr(\Gamma\to A)\eqdef tr(\Gamma)^\bullet\to tr(A)$. Let $tr(Tp(\mathrm{L}))$ be the image of $tr$.

\begin{theorem}\label{embed_lambek}\leavevmode
	\begin{enumerate}
		\item If $\mathrm{L}\vdash \Gamma\to C$, then $\mathrm{HL}\vdash tr(\Gamma\to C)$.
		\item If $\mathrm{HL}\vdash G\to T$ is a derivable graph sequent over $tr(Tp(\mathrm{L}))$, then for some $\Gamma$ and $C$ we have $G\to T=tr(\Gamma\to C)$ (in particular, $G$ has to be a string graph) and $\mathrm{L}\vdash \Gamma\to C$.
	\end{enumerate}
\end{theorem}
This theorem establishes correctness of the embedding $tr$ in a strong way: in particular, it states that if a graph sequent $G\to T$ with its types being images of the Lambek calculus types is derivable, then $G$ is necessarily a string graph (this is not assumed from the beginning), and $G\to T$ corresponds to a derivable string sequent.
\begin{proof}
	The first statement is proved by a straightforward remodelling of a derivation as well as in Theorems \ref{embed_nld}, \ref{embed_lp}; in this section, we prove it in detail while omitting the proof for the rest of the abovementioned theorems.
	
	The proof is by induction on the size of derivation of $\Gamma\to C$ in $\mathrm{L}$. If $p\to p$ is an axiom, then $tr(p\to p)=p^\bullet\to p$ is an axiom of $\mathrm{HL}$.
	
	To prove the induction step, consider the last step of a derivation:
	\begin{itemize}
		\item Case $(/\to)$:
		$$
		\infer[(/\to)]{\Psi, B / A, \Pi, \Delta \to C}{\Psi, B, \Delta \to C & \Pi \to A}
		$$
		By the induction hypothesis, $tr(\Pi \to A)$ and $tr(\Psi, B, \Delta \to C)$ are derivable in $\mathrm{HL}$. Note that $SG=tr(\Psi, B, \Delta)=(tr(\Psi)tr(B)tr(\Delta))^\bullet$ is a string graph with a distinguished edge (call it $e_0$) labeled by $tr(B)$. Then we can construct the following derivation:
		$$
		\infer[(\div\to)]{SG[e_0/D][d_0\eqdef tr(B)\div D][d_1/tr(\Pi)^\bullet] \to tr(C)}{SG \to tr(C) & tr(\Pi)^\bullet \to tr(A)}
		$$
		Here $D$ is the denominator of the type $tr(B/A)$, $E_D=\{d_0,d_1\}$ and $lab(d_0)=\$$, $lab(d_1)=tr(A)$. Finally note that $SG[e_0/D][d_0\eqdef tr(B)\div D][d_1/tr(\Pi)^\bullet]=tr(\Psi, B / A, \Pi, \Delta)$.
		\item Case $(\to/)$:
		$$
		\infer[(\to/)]{\Gamma \to B / A}{\Gamma, A \to B}
		$$
		Here $C=B/A$. By the induction hypothesis, $\mathrm{HL}\vdash tr(\Gamma,A\to B)$. Denote $SG=tr(\Gamma)^\bullet$ and $D=(\$ tr(A))^\bullet$ (where $d_0\in E_D$ is labeled by $\$$). Then
		$$
		\infer[(\to\div)]{SG \to tr(B)\div D}{D[d_0/SG] \to tr(B)}
		$$
		Finishing the proof, we note that $D[d_0/SG]=(\Gamma,A)^\bullet$ and $tr(B)\div D=tr(B/A)$.
		\item Cases $(\backslash\to)$ and $(\to\backslash)$ are treated similarly.
		\item Case $(\cdot\to)$:
		$$
		\infer[(\cdot\to)]{\Psi, A \cdot B, \Delta \to C}{\Psi, A, B, \Delta \to C}
		$$
		is remodeled (applying the induction hypothesis) as follows (where the $tr(A\cdot B)$-labeled edge in $tr(\Psi,A\cdot B,\Delta)$ is denoted as $e_0$):
		$$
		\infer[(\times\to)]{tr(\Psi, A\cdot B, \Delta)^\bullet\to tr(C)}{tr(\Psi, A\cdot B, \Delta)^\bullet[e_0/(tr(A,B))^\bullet] \to tr(C)}
		$$
		\item Case $(\to\cdot)$:
		$$
		\infer[(\to\cdot)]{\Psi, \Delta \to A \cdot B}{\Psi \to A & \Delta \to B}
		$$
		is converted into
		$$
		\infer[(\to\times)]{tr(\Psi, \Delta)^\bullet \to tr(A \cdot B)}{tr(\Psi)^\bullet \to tr(A) & tr(\Delta)^\bullet \to tr(B)}
		$$
	\end{itemize}
	The second statement is of more interest since we know nothing about $G$ at first. The proof again is by induction on length of the derivation. If $G=p^\bullet$ and $C=p$, then obviously $G\to C=tr(p\to p)$.
	\\
	For the induction step consider the last step of a derivation in $\mathrm{HL}$. Below $A,B$ are some types belonging to $Tp(\mathrm{L})$.
	\begin{itemize}
		\item Case $(\div\to)$: after the application of this rule a type of the form $tr(A)\div D$ has to appear. Note that $D$ is either of the form $(\$tr(B))^\bullet$ or of the form $(tr(B)\$)^\bullet$ for some $B$. Let $E_D=\{d_0,d_1\}$ and let $lab(d_0)=\$,lab(d_1)=tr(B)$. Then the application of this rule is of the form
		$$
		\infer[(\div\to)]{H[e/D][d_0\eqdef tr(A)\div D][d_1/H_1]\to T}{H\to T & H_1\to tr(B)}
		$$
		By the induction hypothesis, $H\to T=tr(\Gamma,A,\Delta\to C)$ (since $lab(e)=tr(A)$) and $H_1\to tr(B)=tr(\Pi\to B)$. Therefore, depending on structure of $D$, we obtain that $H[e/D][d_0\eqdef tr(A)\div D][d_1/H_1]\to T$ equals either $tr(\Gamma,\Pi,B\backslash A, \Delta\to C)$ or $tr(\Gamma,A/B,\Pi,\Delta\to C)$, which completes this case.
		\item Case $(\to\div)$: 
		$$
		\infer[(\to\div)]{H\to tr(A)\div D}{D[d_0/H]\to tr(A)}
		$$
		By the induction hypothesis, $D[d_0/H]\to tr(A)$ corresponds to a sequent of the Lambek calculus via $tr$. Again, $D$ is of one of the following forms: $(\$tr(B))^\bullet$ or $(tr(B)\$)^\bullet$. Then the only possibility for $H$ is to be a string graph: $H\eqdef (tr(\Pi))^\bullet$. Thus, $D[d_0/H]$ equals either $tr(\Pi,B)^\bullet$ or $tr(B,\Pi)^\bullet$, and we can model this step in the Lambek calculus by means of $(\to/)$ or $(\to\backslash)$ resp.
		\item Case $(\times\to)$: by the induction hypothesis, a premise has to be of the form $tr(\Pi\to C)$. Then this premise is obtained from a conclustion by replacing an edge labeled by $tr(A\cdot B)$ by a subgraph of the form $(tr(A)tr(B))^\bullet$. This implies that $\Pi=\Gamma,A,B,\Delta$, and that this step can be modeled in $\mathrm{L}$ with the rule $(\cdot\to)$.
		\item Case $(\to\times)$: by the induction hypothesis, all antecedents of premises in this rule are string graphs. Since a succedent of the conclusion has the form $tr(A\cdot B)$, there are two premises, they have antecedents $tr(\Gamma)^\bullet$ and $tr(\Delta)^\bullet$ resp. and they are substituted in $tr(A\cdot B)$. This yields that $G\to T=tr(\Gamma,\Delta\to A\cdot B)$ and that this rule corresponds to $(\to\cdot)$, as expected.
	\end{itemize}
\end{proof}
Look also at Examples \ref{ex_to_div} and \ref{ex_times_to}, which reflect some parts of the proof.

In the remaining subsections we consider different variants of the Lambek calculus and show how they are embedded in $\mathrm{HL}$.

\subsection{Embedding of $\mathrm{LP}$}\label{sec_embed_lp}
$\mathrm{LP}$ is $\mathrm{L}$ enriched with the additional permutation rule:
$$
\infer[(P).]{\Gamma,A,B,\Delta\to C}{\Gamma,B,A,\Delta\to C}
$$
This rule is not logical but structural; however, it can be modeled in $\mathrm{HL}$ by such graphs where permutations of edges lead to isomorphic graphs. One of the ways of doing this is by using the following translation function $tr_P$:
\begin{itemize}
	\item $tr_{\mathrm{P}}(p)=p$;
	\item $tr_{\mathrm{P}}(A/B)=tr_{\mathrm{P}}(B\backslash A)=tr_{\mathrm{P}}(A)\div\left({\tikz[baseline=.1ex]{
			\node (V) {};
			\node[hyperedge, above=-3mm of V] (E1) {$\$$};
			\node[node,right= 6mm of E1,label=below:{\scriptsize $(1)$}] (N1) {};
			\node[hyperedge,right= 6mm of N1] (E2) {$\;tr_P(B)\;$};
			\draw[-,black] (E1) -- node[above] {\scriptsize 1} (N1);
			\draw[-,black] (N1) -- node[above] {\scriptsize 1} (E2);
	}}\right)$;
	\item $tr_{\mathrm{P}}(A\cdot B)=\times\left(
	{\tikz[baseline=.1ex]{
			\node (V) {};
			\node[hyperedge, above=-3mm of V] (E1) {$\;tr_P(A)\;$};
			\node[node,right= 6mm of E1,label=below:{\scriptsize $(1)$}] (N1) {};
			\node[hyperedge,right=6mm of N1] (E2) {$\;tr_P(B)\;$};
			\draw[-,black] (E1) -- node[above] {\scriptsize 1} (N1);
			\draw[-,black] (N1) -- node[above] {\scriptsize 1} (E2);
	}}
	\right)$.	
\end{itemize}
If $\Gamma=T_1,\dots,T_n$ is a sequence of types, then $tr_{\mathrm{P}}(\Gamma)\eqdef \langle \{v_0\},\{e_i\}_{i=1}^n,att,lab,v_0\rangle$ where $att(e_i)=v_0$, $lab(e_i)=tr_{\mathrm{P}}(T_i)$. As before, $tr_{\mathrm{P}}(\Gamma\to A)\eqdef tr_{\mathrm{P}}(\Gamma)\to tr_{\mathrm{P}}(A)$.
\begin{example}
	$p, q\to q\cdot p$ turns into the sequent
	$\;
	{\tikz[baseline=.1ex]{
			\node (V) {};
			\node[hyperedge, above=-3mm of V] (E1) {$p$};
			\node[node,right= 4mm of E1,label=below:{\scriptsize $(1)$}] (N1) {};
			\node[hyperedge,right= 4mm of N1] (E2) {$q$};
			\draw[-,black] (E1) -- node[above] {\scriptsize 1} (N1);
			\draw[-,black] (N1) -- node[above] {\scriptsize 1} (E2);
	}}
	\to \times\left(
	{\tikz[baseline=.1ex]{
			\node (V) {};
			\node[hyperedge, above=-3mm of V] (E1) {$q$};
			\node[node,right= 4mm of E1,label=below:{\scriptsize $(1)$}] (N1) {};
			\node[hyperedge,right= 4mm of N1] (E2) {$p$};
			\draw[-,black] (E1) -- node[above] {\scriptsize 1} (N1);
			\draw[-,black] (N1) -- node[above] {\scriptsize 1} (E2);
	}}
	\right)
	$.
\end{example}
\begin{theorem}\label{embed_lp}\leavevmode
	\begin{enumerate}
		\item If $\mathrm{LP}\vdash \Gamma\to C$, then $\mathrm{HL}\vdash tr_{\mathrm{P}}(\Gamma\to C)$.
		\item If $\mathrm{HL}\vdash G\to T$ is a derivable graph sequent over $tr_P(Tp(\mathrm{L}))$, then $G\to T=tr_{\mathrm{P}}(\Gamma\to C)$ for some $\Gamma$ and $C$ and $\mathrm{LP}\vdash \Gamma\to C$.
	\end{enumerate}
\end{theorem}

\subsection{Embedding of $\mathrm{NL\diamondsuit}$}
$\mathrm{NL\diamondsuit}$ presented in \cite{Moortgat96} is the variant of $\mathrm{L}$ enriched with unary modalities; it lacks structural rules. Types of $\mathrm{NL\diamondsuit}$ are built from primitive types using $\backslash,/,\cdot$ and two unary operators $\diamondsuit,\square$. This variant of the Lambek calculus is nonassociative: antecedents of sequents are considered to be bracketed structures, defined inductively as follows: $\mathcal{T}\eqdef Tp(\mathrm{NL\diamondsuit})\,|\,(\mathcal{T},\mathcal{T})\,|\,(\mathcal{T})^\diamond$. Sequents then are of the form $\Gamma\to A$ where $\Gamma\in\mathcal{T}$ and $A\in Tp(\mathrm{NL\diamondsuit})$.

Rules for $\backslash,\cdot,/$ are similar to those for $\mathrm{L}$ with the difference that they cannot act through brackets
(here $\Gamma[\Delta]$ denotes the term $\Gamma$ containing a distinguished occurence of the subterm $\Delta$):
$$
\infer[(\backslash\to)]{\Gamma[(\Pi, A \backslash B)] \to C}{\Pi \to A & \Gamma[B]\to C}
\qquad
\infer[(\to\backslash)]{\Pi \to A \backslash B}{(A, \Pi) \to B}
\qquad
\infer[(\cdot\to)]{\Gamma[A \cdot B] \to C}{\Gamma[(A,B)] \to C}
$$
$$
\infer[(/\to)]{\Gamma[(B / A, \Pi)]\to C}{\Pi \to A & \Gamma[B] \to C}
\qquad
\infer[(\to/)]{\Pi \to B / A}{(\Pi, A) \to B}
\qquad
\infer[(\to\cdot)]{(\Gamma, \Delta) \to A \cdot B}{\Gamma \to A & \Delta \to B}
$$
The following rules for $\square,\diamondsuit$ are added (where $\Gamma[\Delta]$ denotes the term $\Gamma$ containing a distinguished occurence of the subterm $\Delta$):
$$
\infer[(\diamondsuit\to)]{\Gamma[\diamondsuit A]\to B}{\Gamma[(A)^\diamond]\to B}
\quad
\infer[(\square\to)]{\Gamma[(\square A)^\diamond]\to B}{\Gamma[A]\to B}
\quad
\infer[(\to\diamondsuit)]{(\Gamma)^\diamond\to \diamondsuit A}{\Gamma\to A}
\quad
\infer[(\to\square)]{\Gamma\to \square A}{(\Gamma)^\diamond\to A}
$$
Moortgat notices in \cite{Moortgat96} that $\diamondsuit$ and $\square$ are ``truncated forms of product and implication''. Below we show that this statement can be understood literally: after we embed $\mathrm{NL\diamondsuit}$ in $\mathrm{HL}$, unary connectives become special cases of $\times$ and $\div$. Let us fix primitive types $p_\diamond,p_{l},p_r\in Pr$ (with type equal to 2). Consider the following graphs with $X,Y$ being parameters:
$$Br(X,Y)={\tikz[baseline=.1ex]{
			\node (V) {};
			\node[hyperedge,above=-2mm of V] (E1) {$X$};
			\node[node,right=5mm of E1] (N1) {};
			\node[node,right=10mm of N1, label=above:{\scriptsize $(1)$}] (N2) {};
			\node[node,right=10mm of N2] (N3) {};
			\node[hyperedge,right=5mm of N3] (E2) {$Y$};
			
			\draw[-,black] (E1) -- node[above] {\scriptsize 1} (N1);
			\draw[-,black] (N3) -- node[above] {\scriptsize 1} (E2);
			\draw[>=stealth,->,black] (N2) -- node[above] {$p_l$} (N1);
			\draw[>=stealth,->,black] (N2) -- node[above] {$p_r$} (N3);
	}}
\qquad\qquad
Diam(X)=\mbox{{\tikz[baseline=.1ex]{
			\node (V) {};
			\node[node,above=0.5mm of V, label=left:{\scriptsize $(1)$}] (N2) {};
			\node[node,right=10mm of N2] (N3) {};
			\node[hyperedge,right=5mm of N3] (E2) {$X$};
			
			\draw[-,black] (N3) -- node[above] {\scriptsize 1} (E2);
			\draw[>=stealth,->,black] (N2) -- node[above] {$p_\diamond$} (N3);
	}}}$$
	Let $e_1$ and $e_2$ be the $X$-labeled and the $Y$-labeled edge resp. in $Br(X,Y)$, and let $e$ be the $X$-labeled edge in $Diam(X)$. 
	
	We introduce the following translation function $tr_{\diamond}$:
	\begin{multicols}{2}
	\begin{itemize}
	\item $tr_\diamond(p)=p,\quad p\in Pr,type(p)=1;$
	\item $tr_\diamond((A/B))=tr_\diamond(A)\div\left( Br(\$,tr_\diamond(B))\right)$;
	\item $tr_\diamond((B\backslash A))=tr_\diamond(A)\div\left( Br(tr_\diamond(B),\$)\right)$;
	
	\item $tr_\diamond(\square (A))=tr_\diamond(A)\div\left(Diam(\$)\right)$;
	\item $tr_\diamond((A\cdot B))=\times\left(Br(tr_\diamond(A),tr_\diamond(B))\right)$;
	\item $tr_\diamond(\diamondsuit (A))=\times\left(Diam(tr_\diamond(A))\right)$.
\end{itemize}
\end{multicols}
If $\Gamma,\Delta$ are sequences of types, then $tr_{\diamond}((\Gamma,\Delta))\eqdef Br(X,Y)[e_1/tr_{\diamond}(\Gamma),e_2/tr_{\diamond}(\Delta)]$. Similarly, $tr_{\diamond}((\Gamma)^\diamond)\eqdef Diam(X)[e/tr_{\diamond}(\Gamma)]$. Finally, $tr_{\diamond}(\Gamma\to A)\eqdef tr_{\diamond}(\Gamma)\to tr_{\diamond}(A)$.
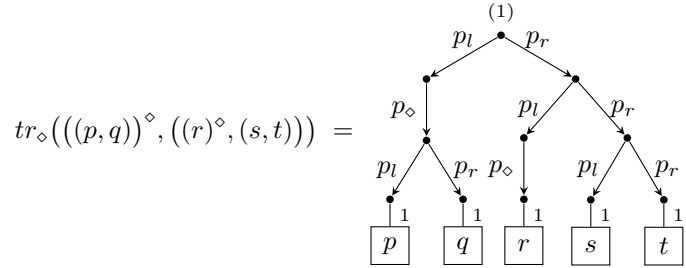
\begin{example}
		$$
		tr_\diamond\big(\big((p,q)\big)^\diamond,\big((r)^\diamond,(s,t)\big)\big)\;=\;\mbox{{\tikz[baseline=.1ex]{
					\node (V) {};
					\node[node, above=12mm of V, label=above:{\scriptsize $(1)$}] (N1) {};
					\node[node, below left=5mm and 9mm of N1] (N2) {};
					\node[node, below right=5mm and 9mm of N1] (N3) {};
					\node[node, below= 7mm of N2] (N4) {};
					\node[node, below left=7mm and 4mm of N4] (N5) {};
					\node[node, below right=7mm and 4mm of N4] (N6) {};
					\node[node, below left=7mm and 6mm of N3] (N7) {};
					\node[node, below right=7mm and 6mm of N3] (N8) {};
					\node[node, below= 7mm of N7] (N9) {};
					\node[node, below left=7.4mm and 4mm of N8] (N10) {};
					\node[node, below right=7.4mm and 4mm of N8] (N11) {};

					\node[hyperedge,below=3mm of N5] (E1) {$p$};
					\node[hyperedge,below=3mm of N6] (E2) {$q$};
					\node[hyperedge,below=3mm of N9] (E3) {$r$};
					\node[hyperedge,below=3mm of N10] (E4) {$s$};
					\node[hyperedge,below=3mm of N11] (E5) {$t$};
					
					\draw[-,black] (N5) -- node[right] {\scriptsize 1} (E1);
					\draw[-,black] (N6) -- node[right] {\scriptsize 1} (E2);
					\draw[-,black] (N9) -- node[right] {\scriptsize 1} (E3);
					\draw[-,black] (N10) -- node[right] {\scriptsize 1} (E4);
					\draw[-,black] (N11) -- node[right] {\scriptsize 1} (E5);

					\draw[->,black] (N1) -- node[above] {$p_l$} (N2);
					\draw[->,black] (N1) -- node[above] {$p_r$} (N3);
					\draw[->,black] (N2) -- node[left] {$p_\diamond$} (N4);
					\draw[->,black] (N4) -- node[left] {$p_l$} (N5);
					\draw[->,black] (N4) -- node[right] {$p_r$} (N6);
					\draw[->,black] (N3) -- node[left] {$p_l$} (N7);
					\draw[->,black] (N3) -- node[right] {$p_r$} (N8);
					\draw[->,black] (N7) -- node[left] {$p_\diamond$} (N9);
					\draw[->,black] (N8) -- node[left] {$p_l$} (N10);
					\draw[->,black] (N8) -- node[right] {$p_r$} (N11);
		}}}
		$$
\end{example}
\begin{theorem}\label{embed_nld}\leavevmode
	\begin{enumerate}
		\item If $\mathrm{NL\diamondsuit}\vdash \Gamma\to C$, then we have $\mathrm{HL}\vdash tr_{\diamond}(\Gamma\to C)$.
		\item If $\mathrm{HL}\vdash G\to T$ is a derivable graph sequent over $tr_\diamond(Tp(\mathrm{NL\diamondsuit}))\cup\{p_l,p_r,p_\diamond\}$, then $G\to T=tr_{\diamond}(\Gamma\to C)$ for some $\Gamma$ and $C$ and $\mathrm{NL\diamondsuit}\vdash \Gamma\to C$.
	\end{enumerate}
\end{theorem}
\begin{remark}
	The nonassociative Lambek calculus $\mathrm{NL}$ can be embedded in $\mathrm{HL}$ as well: it suffices not to consider $\diamondsuit,\square$ and graphs with $p_\diamond$-labeled edges in the above construction. The nonassociative Lambek calculus with permutation $\mathrm{NLP}$ can be embedded in $\mathrm{HL}$ similarly with the only difference that we replace $p_l$ and $p_r$ by a single distinguished type $p_{br}$. The multimodal calculus introduced in \cite{Moortgat96} can also be modeled with this construction, e.g., by introducing indexed primitive types $p_l^i$, $p_r^i$, $p_\diamond^i$ for different modes.
\end{remark}

\subsection{Embedding of $\mathrm{L}_{\mathbf{1}}^\ast$}\label{sec_embed_l1}
In the string case there is a variant of L where empty antecedents are allowed, and there is an additional type $\mathbf{1}$, called the unit. There are one axiom and one rule for it:
$$
\infer[]{\to\mathbf{1}}{}\qquad
\infer[(\mathbf{1}\to)]{\Gamma,\mathbf{1},\Delta\to A}{\Gamma,\Delta\to A}
$$
This extension of L is called the Lambek calculus with the unit and it is denoted by $\mathrm{L}_{\mathbf{1}}^\ast$. Certainly, we wish this calculus to be embedded in HL as well. Our definition of a string graph (Definition \ref{def_str_gr}), though, does not include the case of an empty string so we cannot directly use $tr$ from Section \ref{sec_embed_lambek}. However, if we allow coincidences of external nodes or of attachment nodes, then we define a function $tr_{\mathbf{1}}$ similarly to $tr$ from Section \ref{sec_embed_lambek}; we add the following definition:
$$tr_{\mathbf{1}}(\mathbf{1})\eqdef \times(\langle\{v_0\},\emptyset,\emptyset,\emptyset,v_0v_0\rangle)=\times\left(
	\mbox{{\tikz[baseline=.1ex]{
				\node (R) {};
				\node[node, above=-1mm of R] (N) {};
				\node[left=0mm of N] {\scriptsize $(1)$};
				\node[right=0mm of N] {\scriptsize $(2)$};
	}}}\right)$$
Let us also extend $tr_{\mathbf{1}}$ to sequents as follows: $tr_{\mathbf{1}}(\Gamma\to A)=tr_{\mathbf{1}}(\Gamma)^\bullet\to tr_{\mathbf{1}}(A)$. If $\Gamma$ is empty, we put $(\Lambda)^\bullet\eqdef \langle\{v_0\},\emptyset,\emptyset,\emptyset,v_0v_0\rangle$ (this is consistent with Definition \ref{def_str_gr}).
\begin{theorem}\label{embed_l1}
	Let $\Gamma\to C$ be a sequent over $Tp_{\mathbf{1}}$.
	Then $\mathrm{L}_{\mathbf{1}}^\ast\vdash \Gamma\to C\;\Leftrightarrow\;\mathrm{HL}\vdash tr_{\mathbf{1}}(\Gamma\to C)$. 
\end{theorem}
This theorem is weaker than previous embedding theorems: we restrict our consideration to string graphs only. Generally, if a derivable graph sequent $G\to A$ is over the set $tp_{\mathbf{1}}(Tp_{\mathbf{1}}^\ast)$, then $G$ is not necessarily a string graph: it may contain loops.
\subsection{Embedding of $\mathrm{L}^\Rev$}
The Lambek calculus with the reversal operation (studied e.g. in \cite{Kuznetsov12}) is obtained from $\mathrm{L}$ by adding a unary connective $^\Rev$ and the following rules (where $(A_1,\dots,A_n)^\Rev=A_n^\Rev,\dots,A_1^\Rev$):
$$
\infer[(\,^\Rev\to\,^\Rev)]{\Gamma^\Rev\to C^\Rev}{\Gamma\to C}
\qquad
\infer[(\,^{\Rev\Rev}\to)]{\Gamma,A,\Delta\to C}{\Gamma,A^{\Rev\Rev},\Delta\to C}
\qquad
\infer[(\to\,^{\Rev\Rev})]{\Pi\to C}{\Pi\to C^{\Rev\Rev}}
$$
The cut rule is also included in this calculus (see Section \ref{sec_cut}).

For the set $Tp(\mathrm{L}^\Rev)$ of types in $\mathrm{L}^\Rev$ we introduce the function $tr_{\Rev}$. Its inductive definition coincides with that of $\mathrm{L}$ from Section \ref{sec_embed_lambek} on primitive types and on types of the form $B\backslash A, A\cdot B, A/B$. We extend it to $^\Rev$ as follows:
$$
tr_\Rev(A^\Rev)\eqdef \times\left(\mbox{
	{\tikz[baseline=.1ex]{
			\node[] (R1) {};
			\node[node,above=-2mm of R1, label=left:{\scriptsize $(1)$}] (N1) {};
			\node[node,right=14mm of N1, label=right:{\scriptsize $(2)$}] (N2) {};
			\draw[>=stealth,->,black] (N2) -- node[above] {$tr_\Rev(A)$} (N1);
	}}
}\right)
$$
As usually, $tr_\Rev(\Gamma\to A)\eqdef tr_\Rev(\Gamma)^\bullet\to tr_\Rev(A)$. 
\begin{theorem}\label{embed_lr}\leavevmode
	Let $\Gamma\to C$ be over $\mathrm{L}^\Rev$. Then $\mathrm{L}^\Rev\vdash \Gamma\to C\Leftrightarrow\mathrm{HL}\vdash tr_\Rev(\Gamma\to C)$.
\end{theorem}
\section{Properties of $\mathrm{HL}$}\label{sec_prop}
We start with
\begin{proposition}\label{TtoT}
	$\mathrm{HL}\vdash T^\bullet\to T$ for all types $T$.
\end{proposition}
\begin{proof}
	Induction on size of $T$. If $T$ is primitive, then $T^\bullet\to T$ is an axiom.
	\\
	If $T=N\div D$ and $E_D=\{e_0,\dots,e_k\}$ where $lab(e_0)=\$$, then
	$$
	\infer[(\to\div)]{(N\div D)^\bullet\to N\div D}{
		\infer[(\div\to)]{D[e_0:=N\div D]\to N}{N^\bullet\to N & lab(e_1)^\bullet\to lab(e_1) & \dots & lab(e_k)^\bullet\to lab(e_k)}
	}
	$$
	All the above sequents are derivable by induction hypothesis.
	\\
	If $T=\times(M)$ and $E_M=\{e_1,\dots,e_l\}$, then
	$$
	\infer[(\times\to)]{(\times(M))^\bullet\to \times(M)}{
		\infer[(\to\times)]{M\to \times(M)}{lab(e_1)^\bullet\to lab(e_1) & \dots & lab(e_l)^\bullet\to lab(e_l)}
	}
	$$
	Again, we apply induction hypothesis.
\end{proof}
Therefore, the axiom $p^\bullet\to p$ where primitive types are considered can be replaced by more general one $T^\bullet\to T$ for all types; this does not change the set of derivable sequents.

\subsection{The cut elimination}\label{sec_cut}
One of fundamental properties of $\mathrm{L}$ is admissibility of the following cut rule:
$$
\infer[(\mathrm{cut})]{\Gamma\;\Pi\;\Delta\to B}{\Pi\to A & \Gamma\;A\;\Delta\to B}
$$
For $\mathrm{L}$ this was proved by Lambek in \cite{Lambek58}. This rule can be naturally extended to $\mathrm{HL}$ as follows. Let $H\to A, G\to B$ be graph sequents, $e_0\in E_G$ be an edge, and $lab(e_0)$ be equal to $A$. Then
$$
\infer[(\mathrm{cut})]{G[e_0/H]\to B}{H\to A & G\to B}
$$
\begin{theorem}[cut elimination]\label{cut}
	If $F\to B$ is derivable in $\mathrm{HL}$ enriched with $(\mathrm{cut})$, then it is derivable in $\mathrm{HL}$.
\end{theorem}
This theorem directly implies
\begin{proposition}\label{prop_reversibility}\leavevmode
	\begin{enumerate}
		\item If $\mathrm{HL}\vdash H\to C$ and $e_0\in E_H$ is labeled by $\times(M)$, then $\mathrm{HL}\vdash H[e_0/M]\to C$.
		\item If $\mathrm{HL}\vdash H\to N\div D$ and $e_0\in E_D$ is labeled by $\$$, then $\mathrm{HL}\vdash D[e_0/H]\to N$.
	\end{enumerate}
\end{proposition}
\begin{proof}\leavevmode
	\begin{enumerate}
		\item Use the cut rule as below:
		$$
		\infer[(\mathrm{cut})]{H[e_0/M]\to C}{M\to \times(M) & H\to C}
		$$
		Derivability of $M\to\times(M)$ is trivial.
		\item Use the cut rule as below:
		$$
		\infer[(\mathrm{cut})]{D[e_0/H]\to N}{H\to N\div D & D[e_0\eqdef N\div D]\to N}
		$$
		Derivability of $D[e_0\eqdef N\div D]\to N$ is trivial as well (using $(\div\to)$ we come up with premises $N^\bullet\to N$ and $lab_D(e)^\bullet\to lab_D(e),\; e\ne e_0$).
	\end{enumerate}
\end{proof}
\subsection{Wolf lemma}
\begin{definition}
	Let $A$ be a type, and let $B$ be its distinguished subtype. We say that $B$ is a top occurrence within $A$ if one of the following holds:
	\begin{enumerate}
		\item $A=B$;
		\item $A=\times(M)$ and $\exists e_0\in E_M$ such that $B$ is a top occurrence within $lab(e_0)$;
		\item $A=N\div D$ and $B$ is a top occurrence within $N$. 
	\end{enumerate}
\end{definition}
\begin{example}
	In Example \ref{ex_types} $q$ and $p$ are top occurrences within $A_4$, and $r$ and $s$ are not.
\end{example}
\begin{definition}\label{def_lonely}
	A primitive type $p$ is said to be lonely in a type $A$ if for each top occurrence of $p$ within $A$ there is a subtype $\times(M)$ of $A$ such that $|E_M|\ge 2$ and for some $e_0\in E_M$ $lab(e_0)=p$ is that top occurrence.
\end{definition}
\begin{example}
	In Example \ref{ex_types} $p,s,r$ are lonely in $A_4$, and $q$ is not.
\end{example}
\begin{definition}
	A type $A$ is called skeleton if $A=\times(M)$, $E_M=\emptyset$ and $|ext_M|=|V_M|$.
\end{definition}
\begin{lemma}[wolf lemma]\label{wolf1}
	Let $p\in Pr$ be lonely in $\times(H)$ and let $\times(H)$ not contain skeleton subtypes. Then $\mathrm{HL}\not\vdash H\to p$.
\end{lemma}
\begin{proof}
	The proof is ex falso: assume that $\mathrm{HL}\vdash H\to p$. If a derivation contains an axiom only, then $H=p^\bullet$, which contradicts loneliness of $\times(H)$.
	\\
	Let a derivation include more than one step. There has to be an axiom of the form $p^\bullet\to p$ in this derivation where $p$ is the same as the succedent in $H\to p$. Now it suffices to notice that, however, no rule can be infered to $p^\bullet\to p$ in this derivation.
	\\
	Indeed, if $(\div\to)$ is applied to it, then this step is of the form (notation is like in Section \ref{sec_axioms_rules})
	$$
	\infer[(\div\to)]{G\to p}{p^\bullet\to p & H_1\to lab(d_1) & \dots & H_k\to lab(d_k)}
	$$
	where $G=p^\bullet[e/D][d_0\eqdef N\div D][d_1/H_1,\dots,d_k/H_k]$ and $e$ is the only edge of $p^\bullet$ labeled by $N$. This implies that $N$ has to equal $p$. Consequently, $\times(H)$ contains $p\div D$ as a top occurrence, which contradicts that $p$ is lonely in $\times(H)$. Therefore, this is impossible.
	\\
	Let $(\times\to)$ be infered to $p^\bullet\to p$. Then this step can be presented in the form
	$$
	\infer[(\times\to)]{P\to p}{p^\bullet\to p}
	$$
	Here $P$ is such a graph that $p^\bullet=P[e/F]$ for $e\in E_P$ being labeled by $\times(F)$. If $F$ contains the $p$-labeled edge, then $p$ is not lonely since $\times(F)$ is a top occurrence within $\times(H)$. The remaining option is that $F$ does not contain edges. Note that all nodes in $F$ in such a case have to be external since otherwise isolated node would appear in the premise. Thus, $\times(F)$ is skeleton which contradicts the assumption of the lemma.	
\end{proof}
We will use its corollary:
\begin{corollary}\label{wolf}
	Let $\mathcal{T}$ be a set of types such that for each $T\in\mathcal{T}$ $T$ does not have skeleton subtypes and $p$ is lonely in $T$. Let $\mathrm{HL}\vdash H\to p$ for $H\in\mathcal{H}(\mathcal{T}\cup \{p\})$. Then $H=p^\bullet$.
\end{corollary}
\begin{proof}
	$\times(H)$ does not have skeleton subtypes. Thus, according to Lemma \ref{wolf1}, $p$ is not lonely in $\times(H)$. This means that there is a top occurrence of $p$ within $\times(H)$ for which Definition \ref{def_lonely} does not hold. Let $E_H=\{e_1,\dots,e_n\}$. If this occurrence is a proper subtype of some type $T=lab(e_k)$, then $p$ is not lonely within $T$, which contradicts $T\in\mathcal{T}$. Thus for some $k$ $lab(e_k)=p$. In order for $p$ to be lonely within $\times(H)$, $|E_H|$ necessarily equals 1. This implies that $H\to p$ does not contain $\div$ or $\times$ types, which allows us to draw a conclusion that $H\to p$ is an axiom and $H=p^\bullet$. 
\end{proof}

\subsection{Algorithmic complexity}
In the string case the following theorem was proved by Pentus in \cite{Pentus06}:
\begin{theorem}
	The derivability problem in $\mathrm{L}$ is NP-complete.
\end{theorem}
The derivability problem for $\mathrm{HL}$, obviously, has no less complexity than that for $\mathrm{L}$ (the latter is embedded into the former). It can be shown that complexity is in $\mathrm{NP}$, and thus
\begin{theorem}\label{NP_der}
	The derivability problem in $\mathrm{HL}$ is NP-complete.
\end{theorem}
\subsection{Connection to the First-Order Intuitionistic Logic}\label{ssec_datalog}
Stanislav Kikot pointed out to us that $\mathrm{HL}$ is connected to Datalog\footnote{Datalog is a logic programming language, which can express database queries.} and to the first-order intuitionistic logic ($\FOInt$). To show the main idea of such a connection we provide examples of translating sequents of $\mathrm{HL}$ in formulas of $\FOInt$ (the syntax of these formulas is related to Datalog with embedded implications, see \cite{Bonner89}):
\begin{center}
	\begin{tabular}{|c|c|}
		\hline
		$p\to p$ & $p(x,y)\to p(x,y)$ \\
		\hline
		$tr(p/q,q/r\to p/r)$ & $\exists z\left[\forall a\left(q(z,a)\to p(x,a)\right)\;\&\; \forall b\left(r(y,b)\to q(z,b)\right)\right]\to \forall c \left(r(y,c)\to p(x,c)\right)$ \\
		\hline
		$tr(p/q,q,r\to p\cdot r)$ & $\exists z\exists t\big(\forall a\left(q(z,a)\to p(x,a)\right)\;\&\;q(z,t)\;\&\;r(t,y)\big)\to\exists u(p(x,u)\;\&\;r(u,y))$\\
		\hline
		$
		\mbox{{\tikz[baseline=.1ex]{
					\node (V) {};
					\node[hyperedge, above=-3mm of V] (E1) {$t$};
					\node[node,right= 6mm of E1,label=below:{\scriptsize $(1)$}] (N1) {};
					\node[hyperedge,right=6mm of N1] (E2) {$t$};
					\node[node, above=3mm of N1] (N2) {};
					\node[above=0mm of N2] (N3) {};
					\draw[-,black] (E1) -- node[above] {\scriptsize 1} (N1);
					\draw[-,black] (N1) -- node[above] {\scriptsize 1} (E2);
			}}}\to t
		$ & $\exists y\left[t(x)\;\&\;t(x)\right]\to t(x)$\\
		\hline
	\end{tabular}
\end{center}
This embedding (denoted by $v$) is formally defined in Appendix \ref{sec_datalog_appendix}. It holds that
\begin{theorem}\label{fo-int}
	If $\mathrm{HL}\vdash H\to A$, then $\FOInt\vdash v(H\to A)$.
\end{theorem}
The converse is not true (look e.g. at the last sequent, which is not derivable). 

Summing up, syntactically $\mathrm{HL}$ can be considered as a part of the intuitionistic (or classical, if desired) first-order logic, also connected to logic programming.
\section{Conclusion}
Hypergraph Lambek calculus inherits main principles from $\mathrm{L}$. Many features of $\mathrm{L}$ are preserved by $\mathrm{HL}$ (the cut elimination, but also structural properties and models discussed in the preprint \cite{Pshenitsyn20Preprint}). Besides, different sequential calculi based on concepts related to the Lambek calculus appear to be fragments of $\mathrm{HL}$, and $\mathrm{HL}$ itself can be embedded in the intuitionistic logic. It is also known that $\mathrm{HL}$ forms a basis for categorial grammars, which appear to be stronger than hyperedge replacement grammars (see \cite{Pshenitsyn20Preprint}). In general, $\mathrm{HL}$ is a powerful logical tool, which is interesting to be investigated from different points of view (model-theoretic, proof-theoretic, grammatical).
\bibliography{Hypergraph_Lambek_Calculus_Preprint}

\newpage
\appendixpage
\appendixtocoff
\appendix
\section{Proofs and Sketches of Proofs}\label{sec_proofs}
\subsection{Theorem \ref{embed_nld}}
\begin{proof}[Proof sketch of Theorem \ref{embed_nld}.]\label{embed_nld_proof}
	The first statement is proved by a straightforward induction.
	
	The second statement is also proved by induction on length of a derivation in $\mathrm{HL}$, and in general it is similar to the proof of Theorem \ref{embed_lambek} (see \ref{embed_lambek_proof}). The axiom case is the same (the only difference is that type of primitive types equals 1). Similarly, to prove the induction step, we consider the last rule applied in a derivation of $G\to T$. We note that all premises of this rule have to correspond to sequents in $\mathrm{NL\diamondsuit}$ by the induction hypothesis; then it suffices to note that this last step transforms premises in such a way that the resulting sequent $G\to T$ also corresponds to a derivable sequent in $\mathrm{NL\diamondsuit}$, i.e. $G\to T=tr_\diamond(\Gamma\to C)$; namely, we note that all the transformations are substitutions of trees into trees.
	
	However, one difficulty arises. Let the last rule be, for instance, $(\div\to)$ and let, e.g., $tr_\diamond(A/B)$ appear after its application. Then there are two premises of the form $F_l\to p_{l}$ and $F_r\to p_r$ where $F_l,F_r$ are subgraphs of $G$. Unfortunately, we cannot apply the induction hypothesis to these premises since $p_l,p_r$ are not translations of types. However, here we apply Corollary \ref{wolf}: the only possible sequent (over the considered set of types) with $p_l$ ($p_r$) in the succedent is $p_l^\bullet\to p_l$ ($p_r^\bullet\to p_r$ resp.). Indeed, note that for each type $T$ in the set $\mathcal{T}\eqdef tr_\diamond(Tp(\mathrm{NL\diamondsuit}))\cup\{p_r,p_\diamond\}$ it is true that $T$ does not have skeleton subtypes and that $p_{l}$ is lonely in $T$. Thus we can apply Corollary \ref{wolf} to $F_l\to p_l$ and obtain that $F_l=p_l^\bullet$. This is a desired result: $p_l$-labeled edges can ``interact'' only with $p_l$-labeled edges, and hence they work in a way which corresponds to rules of $\mathrm{NL\diamondsuit}$. The same reasoning works with $p_r$ and $p_\diamond$. 
	
	Using this observation one can straightforwardly show how to transform each step of a derivation in $\mathrm{HL}$ into a step of a derivation in $\mathrm{NL}\diamondsuit$.
\end{proof}

\subsection{Theorem \ref{embed_l1}}\label{embed_l1_proof}
\begin{proof}
	The proof is similar to that of Theorem \ref{embed_lambek}. In both directions we need to use induction on length of a derivation. Let us focus on cases where $\mathbf{1}$ participates in each direction.
	\\
	\begin{enumerate}
		\item Let $\mathrm{L}_{\mathbf{1}}^\ast\vdash \Gamma\to C$. 
		
		If it is the axiom $\to\mathbf{1}$, then it is translated into a graph sequent $(\Lambda)^\bullet\to\times((\Lambda)^\bullet)$.
		
		If this sequent is obtained after the rule application
		$$
		\infer[(\mathbf{1}\to)]{\Phi,\mathbf{1},\Psi\to A}{\Phi,\Psi\to A}
		$$
		then this rule can be remodeled by $(\times\to)$ with the type $tr_{\mathbf{1}}(\mathbf{1})$.
		
		\item Let $\mathrm{HL}\vdash tr_{\mathbf{1}}(\Gamma\to C)$.
		
		If $tr_{\mathbf{1}}(C)=tr_{\mathbf{1}}(\mathbf{1})$ and the last rule is $(\to\times)$, then $tr_{\mathbf{1}}(\Gamma)=(\Lambda)^\bullet$, and this corresponds to the axiom case $\to\mathbf{1}$.
		
		If the last rule is $(\times\to)$, and it is applied to a type of the form $tp_{\mathbf{1}}(\mathbf{1})$, then it can be remodeled using the rule $(\mathbf{1}\to)$.
	\end{enumerate}
\end{proof}
\subsection{Theorem \ref{embed_lr}}
Let us start with the following
\begin{definition}
	A \emph{generalized string graph} induced by a string $a_1\dots a_n$ is a graph of the form $\langle\{v_0,\dots,v_n\},\{e_1,\dots,e_n\},att,lab,v_0v_n\rangle$ where $lab(e_i)=a_i$ and either $att(e_i)=v_{i-1}v_i$ or $att(e_i)=v_iv_{i-1}$. It is denoted as $(a_1^{\delta_1}\dots a_n^{\delta_n})^\bullet$ where $\delta_i=0$ if $att(e_i)=v_{i-1}v_i$ and $\delta_i=1$ otherwise.
\end{definition}
\begin{definition}
	We denote by $[A]^0$ the type $A$ and by $[A]^1$ the type $A^\Rev$.
\end{definition}
Let us introduce a function called $tr_\Rev^\prime$ operating on sequents, which assigns a finite set of graph sequents to each sequent in $\mathrm{L}^\Rev$. This is done as follows: if $Seq=A_1,\dots, A_n\to C$ is a sequent and $A_i=[B_i]^{\delta_i}$ for some $B_i$ and $\delta_i\in\{0,1\}$, then $(tr_\Rev(B_1)^{\delta_1}\dots tr_\Rev(B_n)^{\delta_n})^\bullet\to tr_\Rev(C)$ belongs to $tr_\Rev^\prime(Seq)$. Thus $tr^\prime_\Rev(Seq)$ contains graph sequents that differ from $tr_\Rev(Seq)$ by directions of some arrows and absence of corresponding reversal operations.
\begin{example}
	$tr_\Rev^\prime(p^\Rev,q^\Rev,r\to A)$ for some type $A$ consists of 4 graph sequents:
	\begin{enumerate}
		\item $\mbox{	
			{\tikz[baseline=.1ex]{
					\node[node,label=left:{\scriptsize $(1)$}] (N1) {};
					\node[node,right=15mm of N1] (N2) {};
					\node[node,right=15mm of N2] (N3) {};
					\node[node,right=15mm of N3,label=right:{\scriptsize $(2)$}] (N4) {};
					\draw[->,black] (N1) -- node[above] {$tr_\Rev(p^\Rev)$} (N2);
					\draw[->,black] (N2) -- node[above] {$tr_\Rev(q^\Rev)$} (N3);
					\draw[->,black] (N3) -- node[above] {$r$} (N4);
		}}}\to tr_\Rev(A)$
		\item $\mbox{	
			{\tikz[baseline=.1ex]{
					\node[node,label=left:{\scriptsize $(1)$}] (N1) {};
					\node[node,right=15mm of N1] (N2) {};
					\node[node,right=15mm of N2] (N3) {};
					\node[node,right=15mm of N3,label=right:{\scriptsize $(2)$}] (N4) {};
					\draw[->,black] (N1) -- node[above] {$tr_\Rev(p^\Rev)$} (N2);
					\draw[->,black] (N3) -- node[above] {$q$} (N2);
					\draw[->,black] (N3) -- node[above] {$r$} (N4);
		}}}\to tr_\Rev(A)$
		\item $\mbox{	
			{\tikz[baseline=.1ex]{
					\node[node,label=left:{\scriptsize $(1)$}] (N1) {};
					\node[node,right=15mm of N1] (N2) {};
					\node[node,right=15mm of N2] (N3) {};
					\node[node,right=15mm of N3,label=right:{\scriptsize $(2)$}] (N4) {};
					\draw[->,black] (N2) -- node[above] {$p$} (N1);
					\draw[->,black] (N2) -- node[above] {$tr_\Rev(q^\Rev)$} (N3);
					\draw[->,black] (N3) -- node[above] {$r$} (N4);
		}}}\to tr_\Rev(A)$
		\item $\mbox{	
			{\tikz[baseline=.1ex]{
					\node[node,label=left:{\scriptsize $(1)$}] (N1) {};
					\node[node,right=15mm of N1] (N2) {};
					\node[node,right=15mm of N2] (N3) {};
					\node[node,right=15mm of N3,label=right:{\scriptsize $(2)$}] (N4) {};
					\draw[<-,black] (N1) -- node[above] {$p$} (N2);
					\draw[<-,black] (N2) -- node[above] {$q$} (N3);
					\draw[->,black] (N3) -- node[above] {$r$} (N4);
		}}}\to tr_\Rev(A)$
	\end{enumerate}
\end{example}
\begin{lemma}\label{lemma_lr}\leavevmode
	\begin{enumerate}
		\item $tr_\Rev(Seq)\in tr^\prime_\Rev(Seq)$;
		\item If $H\to T$ belongs to $tr^\prime_\Rev(Seq)$, then $\mathrm{HL}\vdash tr_\Rev(Seq)$ if and only if $\mathrm{HL}\vdash H\to T$. 
	\end{enumerate}
\end{lemma}
\begin{proof}
	\begin{enumerate}
		\item It suffices to notice that if $Seq=A_1,\dots, A_n\to C$, then one may take $\delta_i=0$ and $B_i=A_i$ for all $i$. In such a case $(B_1^{\delta_1}\dots B_n^{\delta_n})^\bullet\to tr_\Rev(C)=tr_\Rev(Seq)$.
		\item Proposition \ref{prop_reversibility} says that derivability of a sequent $H\to C$ with a type of the form $\times(M)$ on some edge $e_0$ in the antecedent is equivalent to derivability of $H[e_0/M]\to C$. We have already proved that $tr_\Rev(Seq)\in tr^\prime_\Rev(Seq)$, and showed that this corresponds to the case $\delta_i=0$. Now, if in the definition of $tr^\prime_\Rev$ some $\delta_i$ equals $1$, then $A_i=B_i^\Rev$, and an antecedent of the resulting sequent is obtained from that of $tr_\Rev(Seq)$ by replacement of the $i$-th edge labeled by $tr_\Rev(A_i)$ with the edge pointing in the opposite direction and labeled by $tr_\Rev(B_i)$. Proposition \ref{prop_reversibility} implies that this does not affect derivability of the sequent. The same reasonings hold, if several $\delta_i$-s equal 1.
	\end{enumerate}
	\qed
\end{proof}
Theorem \ref{embed_lr} will be proved in the more general case:
\begin{theorem}\label{embed_lr_strong}\leavevmode
	\begin{enumerate}
		\item Let $\Gamma\to C$ be over $\mathrm{L}^\Rev$. Then $\mathrm{L}^\Rev\vdash \Gamma\to C$ if and only if $\mathrm{HL}\vdash tr_\Rev(\Gamma\to C)$.
		\item If $\mathrm{HL}\vdash G\to T$ is a derivable graph sequent over $tr_\Rev(\mathrm{L}^\Rev)$, then $G\to T$ belongs to $tr^\prime_\Rev(\Gamma\to C)$ for some $\Gamma$ and $C$ and $\mathrm{L}^\Rev\vdash \Gamma\to C$.
	\end{enumerate}
\end{theorem}
\begin{proof}
	Firstly, we are going to prove \textbf{the ``only if'' part of the first statement}. This is done by induction on length of the derivation of $\Gamma\to C$. 
	
	The axiom case is trivial. To prove the induction step we look at the last rule applied, apply the induction hypothesis to its premises and then show that the application of this rule in $\mathrm{L}^\Rev$ corresponds to an application of a rule in $\mathrm{HL}$. It suffices to consider only rules where $^\Rev$ actively participates since cases with applications of rules $(/\to)$, $(\backslash\to)$ etc. are completed in the proof of Theorem \ref{embed_lambek}. 
	
	\textbf{Case $(\,^\Rev\to\,^\Rev)$}: $\Gamma\to C=\Delta^\Rev\to B^\Rev$ where $\Delta=B_1,\dots,B_n$ and the last rule is of the form
	$$
	\infer[(\,^\Rev\to\,^\Rev)]{\Delta^\Rev\to B^\Rev}{\Delta\to B}
	$$
	Then $tr_\Rev(\Delta\to B)$ is derivable in HL. Applying the rule $(\to\times)$ to it, we obtain a sequent $Seq=(tr_\Rev(B_n)^1\dots tr_\Rev(B_1)^1)^\bullet\to tr_\Rev(B^\Rev)$. According to Lemma \ref{lemma_lr} this sequent is derivable if and only if $\mathrm{HL}\vdash (tr_\Rev(B_n^\Rev)\dots tr_\Rev(B_1^\Rev))^\bullet\to tr_\Rev(B^\Rev)=tr_\Rev(\Gamma\to C)$, which completes the proof in this case.

	\textbf{Case $(\to\,^{\Rev\Rev})$} being translated to graph sequents can be modelled using the cut rule, which is admissible. Indeed, let
	$$
	\infer[]{\Gamma\to C}{\Gamma\to C^{\Rev\Rev}}
	$$
	be the last rule. Then we know that $\mathrm{HL}\vdash tr_\Rev(\Gamma)^\bullet\to tr_\Rev(C^{\Rev\Rev})$. The sequent $tr_\Rev(C^{\Rev\Rev})^\bullet\to tr_\Rev(C)$ is derivable (the last two steps are applications of $(\times\to)$). Now we apply the cut rule:
	$$
	\infer[(\mathrm{cut})]{
		tr_\Rev(\Gamma)^\bullet\to tr_\Rev(C)
	}{
	tr_\Rev(\Gamma)^\bullet\to tr_\Rev(C^{\Rev\Rev}) & tr_\Rev(C^{\Rev\Rev})^\bullet\to tr_\Rev(C) 
	}
	$$
	Here $\Gamma=A_1\dots A_n$. Informally, we turned over the antecedent of the sequent twice and hence changed nothing.
	
	\textbf{Case $(\,^{\Rev\Rev}\to)$} being translated to graph sequents corresponds to a double application of Proposition \ref{prop_reversibility} (eliminating $\times$ twice changes the direction of an arrow twice, hence changes nothing). 
	
	\textbf{Case $(\mathrm{cut})$} immediately follows from admissibility of the cut rule in $\mathrm{HL}$ (Theorem \ref{cut}).
	
	\textbf{The ``if'' part of the first statement} follows from the second statement: if $G\to T=tr_\Rev(\Gamma\to C)$, then there is a sequent $\Gamma^\prime\to C^\prime$ such that $G\to T$ belongs to $tr^\prime_\Rev(\Gamma^\prime\to C^\prime)$ and $\mathrm{L}^\Rev\vdash \Gamma^\prime\to C^\prime$. The fact that $tr_\Rev(\Gamma\to C)$ lies in $tr^\prime_\Rev(\Gamma^\prime\to C^\prime)$ along with the definition of $tr^\prime$ and structure of generalized string graphs implies that $\Gamma\to C=\Gamma^\prime\to C^\prime$, so $\mathrm{L}^\Rev\vdash \Gamma\to C$ as required.
	
	It remains to prove \textbf{the second statement}. First of all, one observes that 
	\begin{enumerate}
		\item Replacement of an edge in a generalized string graph with a generalized string graph results in a generalized string graph as well;
		\item If we replace an edge in a generalized string graph by some graph $G$ and obtain a generalized string graph, then $G$ is a generalized string graph too.
		\item If we replace an edge in some graph $G$ by a generalized string graph and obtain a generalized string graph, then $G$ is a generalized string graph too.
	\end{enumerate}
	Therefore $G$ from the second statement of the theorem has to be a generalized string graph (since the property of being such a graph holgs for antecedents of axiom sequents and is preserved by all rules); that is, $G=(tr_\Rev(A_1)^{\delta_1}\dots tr_\Rev(A_n)^{\delta_n})^\bullet$ for some $n$, for some $A_i\in Tp(\mathrm{L}^\Rev)$ and for some $\delta_i$.
	
	We proceed with induction on length of the derivation of $G\to T$. The axiom case is trivial. As before, the induction step consists of consideration of the last rule applied in the derivation. 
	
	\textbf{Case 1.} $T=tr_\Rev(A\cdot B)$ and the last rule is $(\to\times)$. Then it is of the form
	$$
	\infer[(\to\times)]{G\to T}{(\Gamma^\prime)^\bullet\to tr_\Rev(A) & (\Delta^\prime)^\bullet\to tr_\Rev(B)}
	$$
	where $\Gamma^\prime,\Delta^\prime=tr_\Rev(A_1)^{\delta_1}\dots tr_\Rev(A_n)^{\delta_n}$ (note that $\Gamma^\prime$ and $\Delta^\prime$ are not regular strings but strings with symbols indexed by $0$ or $1$). By the induction hypothesis, there exist derivable sequents $Seq_1=\Gamma\to A$ and $Seq_2=\Delta\to B$ such that $(\Gamma^\prime)^\bullet\to tr_\Rev(A)\in tr^\prime_\Rev(Seq_1)$ and $(\Delta^\prime)^\bullet\to tr_\Rev(B)\in tr^\prime_\Rev(Seq_2)$. Then $G\to T$ belongs to $tr^\prime_\Rev(\Gamma,\Delta\to A\cdot B)$ and the sequent $\Gamma,\Delta\to A\cdot B$ is derivable in $\mathrm{L}^\Rev$ (using the rule $(\to\cdot)$).
	
	\textbf{Case 2.} $T=tr_\Rev(A/B)$ ($T=tr_\Rev(B\backslash A)$) and the last rule in the derivation is $(\to\div)$. The proof for such cases is similar to that in Case 1: we show how to remodel this production using the rule $(\to/)$ ($(\to\backslash)$ resp.) of the Lambek calculus.
	
	\textbf{Case 3.} $T=tr_\Rev(A^\Rev)$ and the last rule is $(\to\times)$. Then this rule application is of the form
	$$
	\infer[(\to\times)]{(tr_\Rev(A_1)^{\delta_1}\dots tr_\Rev(A_n)^{\delta_n})^\bullet\to tr_\Rev(A^\Rev)}
	{(tr_\Rev(A_n)^{\overline{\delta_n}}\dots tr_\Rev(A_1)^{\overline{\delta_1}})^\bullet\to tr_\Rev(A)}
	$$
	where $\overline{\delta}\eqdef (\delta+1 \mod 2)$. Let us consider a simple but comprehensive example when $n=2$, $\delta_1=0$ and $\delta_2=1$:
	$$
	\infer[(\to\times)]{(tr_\Rev(A_1)^0, tr_\Rev(A_2)^1)^\bullet\to tr_\Rev(A^\Rev)}
	{(tr_\Rev(A_2)^0, tr_\Rev(A_1)^1)^\bullet\to tr_\Rev(A)}
	$$
	One observes that the only sequent $Seq$ such that $(tr_\Rev(A_2)^0, tr_\Rev(A_1)^1)^\bullet\to tr_\Rev(A)$ belongs to $tr_\Rev^\prime(Seq)$ is the sequent $A_2,A_1^\Rev\to A$. By the induction hypothesis, $\mathrm{L}^\Rev\vdash A_2,A_1^\Rev\to A$. Then we derive in $\mathrm{L}^\Rev$ the following:
	$$
	\infer[(^{\Rev\Rev}\to)]{A_1,A_2^\Rev\to A^\Rev}
		 {\infer[(\,^\Rev\to\,^\Rev)]{A_1^{\Rev\Rev},A_2^\Rev\to A^\Rev}{
			A_2,A_1^\Rev\to A
		}
	}
	$$
	Finally, notice that $G\to T$ lies in $tr^\prime_\Rev(A_1,A_2^\Rev\to A^\Rev)$. A similar derivation can be done in a general case: the only difference is that one applies $(^{\Rev\Rev}\to)$ several times.
	
	\textbf{Case 4a.} Let for some $k\in\{1,\dots, n\}$ $\delta_k=0$, $A_k=tr_\Rev(A/B)$ and $$\Pi^\prime=\\tr_\Rev(A_{k+1})^{\delta_{k+1}},\dots, tr_\Rev(A_{l})^{\delta_{l}}$$ for some $k< l\le n$. Let the last rule be $(\div\to)$, in which $tr_\Rev(A_k)$ and $(\Pi^\prime)^\bullet$ participate. Then the rule is of the form
	$$
	\infer[(\div\to)]{(\Gamma^\prime,tr_\Rev(A/B)^0,\Pi^\prime,\Delta^\prime)^\bullet\to T}{(\Gamma^\prime,tr_\Rev(A)^0,\Delta^\prime)^\bullet\to T & (\Pi^\prime)^\bullet\to tr_\Rev(B)}
	$$
	Here we denote by $\Gamma^\prime$ the sequence of indexed types $tr_\Rev(A_{1})^{\delta_{1}},\dots, tr_\Rev(A_{k-1})^{\delta_{k-1}}$ and by $\Delta^\prime$ the sequence $tr_\Rev(A_{l+1})^{\delta_{l+1}},\dots, tr_\Rev(A_{n})^{\delta_{n}}$. By the induction hypothesis, $(\Gamma^\prime,tr_\Rev(A)^0,\Delta^\prime)^\bullet\to T$ belongs to $tr_\Rev^\prime(\Gamma,A,\Delta\to C)$ and $(\Pi^\prime)^\bullet\to tr_\Rev(B)$ belongs to $tr_\Rev^\prime(\Pi\to B)$ for derivable sequents $\Gamma,A,\Delta\to C$ and $\Pi\to B$. Then we apply the rule $(/\to)$ to them and obtain $Seq=\Gamma,A/B,\Pi,\Delta\to C$; finally, notice that $(\Gamma^\prime,tr_\Rev(A/B)^0,\Pi^\prime,\Delta^\prime)^\bullet\to T\in tr^\prime_\Rev(Seq)$.
	
	\textbf{Case 4b.} Let for some $k\in\{1,\dots, n\}$ $\delta_k=1$, $A_k=tr_\Rev(A/B)$, and $$\Pi^\prime=tr_\Rev(A_{l})^{\delta_{l}},\dots, tr_\Rev(A_{k-1})^{\delta_{k-1}}$$ for some $l<k\le n$; let the last rule be $(\div\to)$, in which $tr_\Rev(A_k)$ and $(\Pi^\prime)^\bullet$ participate. Then the rule is of the form
	$$
	\infer[(\div\to)]{(\Gamma^\prime,tr_\Rev(A_l)^{\delta_l},\dots, tr_\Rev(A_{k-1})^{\delta_{k-1}}, tr_\Rev(A/B)^1,\Delta^\prime)^\bullet\to T}{(\Gamma^\prime,tr_\Rev(A)^1,\Delta^\prime)^\bullet\to T & (tr_\Rev(A_{k-1})^{\overline{\delta_{k-1}}}\dots tr_\Rev(A_l)^{\overline{\delta_l}})^\bullet\to tr_\Rev(B)}
	$$
	By the induction hypothesis, $(\Gamma^\prime,tr_\Rev(A)^1,\Delta^\prime)^\bullet\to T$ belongs to $tr^\prime_\Rev(\Gamma, A^\Rev,\Delta\to C)$ and $(tr_\Rev(A_{k-1})^{\overline{\delta_{k-1}}}\dots tr_\Rev(A_l)^{\overline{\delta_l}})^\bullet\to tr_\Rev(B)$ belongs to $tr_\Rev^\prime(\Pi\to B)$ for derivable sequents $\Gamma,A^\Rev,\Delta\to C$ and $\Pi\to B$ (where $\Pi=[A_{k-1}]^{\overline{\delta_{k-1}}},\dots,[A_{l}]^{\overline{\delta_{l}}}$). From $\Pi\to B$ we can derive the following sequent:
	$$
	\infer[(^\Rev\to^\Rev)]{
		\infer[]{
			[A_{l}]^{\delta_{l}},\dots,[A_{k-1}]^{\delta_{k-1}}\to B^\Rev
		}
		{
			([A_{l}]^{\overline{\delta_{l}}})^\Rev,\dots,([A_{k-1}]^{\overline{\delta_{k-1}}})^\Rev\to B^\Rev
		}
	}
	{
		[A_{k-1}]^{\overline{\delta_{k-1}}},\dots,[A_{l}]^{\overline{\delta_{l}}}\to B
	}
	$$
	The last step of a derivation is a series of applications of the rule $(^{\Rev\Rev}\to)$.
	
	Now we compose this sequent with $\Gamma,A^\Rev, \Delta\to C$ as follows:
	$$
	\infer[(\backslash\to)]{
		\Gamma,[A_{l}]^{\delta_{l}},\dots,[A_{k-1}]^{\delta_{k-1}}, B^\Rev\backslash A^\Rev,\Delta\to C
	}{
		\Gamma,A^\Rev,\Delta\to C & [A_{l}]^{\delta_{l}},\dots,[A_{k-1}]^{\delta_{k-1}}\to B^\Rev
	}
	$$
	Finally, we use the fact that $B^\Rev\backslash A^\Rev\leftrightarrow (A/B)^\Rev$. Thus $$Seq= \Gamma,[A_{l}]^{\delta_{l}},\dots,[A_{k-1}]^{\delta_{k-1}}, (A/B)^\Rev,\Delta\to C$$ is derivable in $\mathrm{L}^\Rev$ and $(\Gamma^\prime,tr_\Rev(A_l)^{\delta_l},\dots, tr_\Rev(A_{k-1})^{\delta_{k-1}}, tr_\Rev(A/B)^1,\Delta^\prime)^\bullet\to T\in tr^\prime_\Rev(Seq)$, which we wanted to prove.
	
	\textbf{Cases 5a and 5b} are similar to cases 4a and 4b with the only difference that we consider $A_k=tr_\Rev(B\backslash A)$.
	
	\textbf{Case 6a.} Let for some $k\in\{1,\dots,n\}$ $\delta_k=0$, $A_k=tr_\Rev(A\cdot B)$ and let the last rule be applied to this type. Then this rule is of the form
	$$
	\infer[(\times\to)]{
		(\Gamma^\prime, tr_\Rev(A\cdot B)^0, \Delta^\prime)^\bullet\to T
	}
	{
		(\Gamma^\prime, tr_\Rev(A)^0,tr_\Rev(B)^0, \Delta^\prime)^\bullet\to T
	}
	$$
	As always, we apply the induction hypothesis and then apply the rule $(\cdot\to)$.
	
	\textbf{Case 6b} differs from case 6a in that $\delta_k=1$:
	$$
	\infer[(\times\to)]{
		(\Gamma^\prime, tr_\Rev(A\cdot B)^1, \Delta^\prime)^\bullet\to T
	}
	{
		(\Gamma^\prime, tr_\Rev(B)^1,tr_\Rev(A)^1, \Delta^\prime)^\bullet\to T
	}
	$$
	By the induction hypothesis, $(\Gamma^\prime, tr_\Rev(B)^1,tr_\Rev(A)^1, \Delta^\prime)^\bullet\to T$ belongs to $tr^\prime_\Rev(\Gamma,B^\Rev,A^\Rev,\Delta\to C)$ and $\Gamma,B^\Rev,A^\Rev,\Delta\to C$ is derivable. Then we can derive the following:
	$$
	\infer[(\cdot\to)]{\Gamma,B^\Rev\cdot A^\Rev,\Delta\to C}
	{\Gamma,B^\Rev,A^\Rev,\Delta\to C}
	$$
	Finally, we use that $B^\Rev\cdot A^\Rev\leftrightarrow (A\cdot B)^\Rev$.
	
	\textbf{Case 7a.} Let for some $k\in\{1,\dots,n\}$ $\delta_k=0$, $A_k=tr_\Rev(A^\Rev)$ and let the last rule be applied to this type. Then this rule is of the form
	$$
	\infer[(\times\to)]{
		(\Gamma^\prime, tr_\Rev(A^\Rev)^0, \Delta^\prime)^\bullet\to T
	}
	{
		(\Gamma^\prime, tr_\Rev(A)^1, \Delta^\prime)^\bullet\to T
	}
	$$
	By the induction hypothesis, $(\Gamma^\prime, tr_\Rev(A)^1, \Delta^\prime)^\bullet\to T$ belongs to a translation of the form $tr_\Rev^\prime(\Gamma,A^\Rev,\Delta\to C)$. Note, however, that $(\Gamma^\prime, tr_\Rev(A^\Rev)^0, \Delta^\prime)^\bullet\to T$ also belongs to $tr_\Rev^\prime(\Gamma,A^\Rev,\Delta\to C)$ due to the definition of $tr^\prime_\Rev$, so the sequent $\Gamma,A^\Rev,\Delta\to C$ is the required one.
	
	\textbf{Case 7b}, again, differs from Case 7a in the following: $\delta_k=1$. The last rule then is of the form
	$$
	\infer[(\times\to)]{
		(\Gamma^\prime, tr_\Rev(A^\Rev)^1, \Delta^\prime)^\bullet\to T
	}
	{
		(\Gamma^\prime, tr_\Rev(A)^0, \Delta^\prime)^\bullet\to T
	}
	$$
	We may conclude that $(\Gamma^\prime, tr_\Rev(A)^0, \Delta^\prime)^\bullet\to T$ lies in $tr^\prime_\Rev(\Gamma,A,\Delta\to C)$ for $\Gamma,A,\Delta\to C$ being derivable in $\mathrm{L}^\Rev$. Using this, we construct the following derivation:
	$$
	\infer[(cut)]{
		\Gamma,A^{\Rev\Rev},\Delta\to C
	}
	{
		A^{\Rev\Rev}\to A & \Gamma,A,\Delta\to C
	}
	$$
	The sequent $A^{\Rev\Rev}\to A$ is derivable (using $(\to^{\Rev\Rev})$). Finally, it suffices to notice that $(\Gamma^\prime, tr_\Rev(A^\Rev)^1, \Delta^\prime)^\bullet\to T$ lies in $tr^\prime_\Rev(\Gamma,A^{\Rev\Rev},\Delta\to C)$.
\end{proof}

\subsection{Theorem \ref{cut}}\label{cut_proof}
\begin{definition}
	Size $|T|$ of a type $T$ is defined inductively as follows:
	\begin{enumerate}
		\item $|p|\eqdef 1$ for $p\in Pr$;
		\item If $T=N\div D$ and $E_D=\{d_0,\dots,d_k\}$ with $lab_D(d_0)$ being equal to \$, then $|T|\eqdef |N|+|lab_D(d_1)|+\dots+|lab_D(d_k)|+1$;
		\item If $T=\times(M)$ and $E_M=\{m_1,\dots,m_k\}$, then $|T|\eqdef |lab_M(m_1)|+\dots+|lab_M(m_k)|+1$.
	\end{enumerate}
\end{definition}
\begin{proof}[Proof (of Theorem \ref{cut}).]
	We prove that if $\mathrm{HL}\vdash H\to A$ and $\mathrm{HL}\vdash G\to B$, then $\mathrm{HL}\vdash G[e_0/H]\to B$ where $e_0\in E_G$ and $lab(e_0)=A$ by induction on $|H\to A|+|G\to B|$.
	
	\textbf{Case 1.} $H\to A$ is an axiom $p^\bullet\to p$. Then $G[e_0/H]=G$, so the replacement changes nothing. 
	
	\textbf{Case 2.} $G\to B$ is an axiom $p^\bullet\to p$. Then $A=lab(e_0)=B$, and $G[e_0/H]\to B = H\to A$, so the conclusion coincides with one of the premises.
	
	Let us further call the distinguished type $N\div D$ in rules $(\div\to)$ and $(\to\div)$, and the distinguished type $\times(M)$ in rules $(\times\to)$ and $(\to\times)$ (we mean those from definitions in Section \ref{sec_axioms_rules}) the \emph{major type of the rule}.
	
	\textbf{Case 3.} In $H\to A$, the type $A$ is not the major type of the last rule applied. There are two subcases depending on the type of this rule. 
	
	\textbf{Case 3a.} $(\div\to)$:
	$$
	\infer[(\mathrm{cut})]{G[e_0/H]\to B}{
		\infer[(\div\to)]{H\to A}{
			K\to A & H_1\to T_1 & \dots & H_k\to T_k
		}
		&
		G\to B
	}
	$$
	Here $H$ is obtained from $K$ by replacements using $H_1,\dots,H_k$, as the rule $(\div\to)$ prescribes. Note that we omit some details of rule applications that are not essential here (but their role can be understood from the general structure of the rule).
	\\
	This derivation is transformed as follows:
	$$
	\infer[(\div\to)]{G[e_0/H]\to B}{
		\infer[(\mathrm{cut})]{G[e_0/K]\to B}{
			K\to A & G\to B
		}
		&
		H_1\to T_1 & \dots & H_k\to T_k
	}
	$$
	Now we apply the induction hypothesis to the premises and obtain a cut-free derivation for $G[e_0/H]\to B$. Further the induction hypothesis will be applied in a similar way to the premises appearing in the new derivation process. Sometimes the induction hypothesis will be applied several times (from top to bottom, see Cases 5 and 6); however, this will be always legal.
	
	\textbf{Case 3b} $(\times\to)$. Let $f_0\in E_H$ be labeled by a type $\times(K)$, which apears at the last step of a derivation. Then the remodelling is as follows:
	$$
	\infer[(\mathrm{cut})]{G[e_0/H]\to B}{
		\infer[(\times\to)]{H\to A}{
			H[f_0/K]\to A
		}
		&
		G\to B
	}
	\rightsquigarrow
	\infer[(\times\to)]{G[e_0/H]\to B}{
		\infer[(\mathrm{cut})]{G[e_0/H[f_0/K]]\to A}{
			H[f_0/K]\to A & G\to B
		}
	}
	$$
	Here and further symbols like $\rightsquigarrow$ stand for remodelling a derivation.
	
	\textbf{Case 4.} The type $A$ labeling $e_0$ within $G$ is not the major type in the last rule in the derivation of $G\to B$. Then one repeats the last step of the derivation of $G\to B$ in $G[e_0/H]\to B$ considering $H$ to be an atomic structure acting as $e_0$. Formally, there are five subcases depending on the last rule applied in the derivation of $G\to B$:
	\begin{enumerate}[wide, labelwidth=!, labelindent=0pt]
		\item $(\div\to)$ if one of the subgraphs $H_i$ contains $e_0$:
		$$
		\infer[(\mathrm{cut})]{G[e_0/H]\to B}{
			H\to A
			&
			\infer[(\div\to)]{G\to B}{
				K\to B & H_1\to T_1\;\dots & H_i\to T_i &\dots\; H_k\to T_k
			}
		}
		$$
		Let $H_i$ contain an edge $e_0$; then this derivation is remodeled as follows:
		$$
		\infer[(\div\to)]{G[e_0/H]\to B}{
			K\to B & H_1\to T_1\;\dots & \infer[(\mathrm{cut})]{H_i[e_0/H]\to T_i}{H\to A & H_i\to T_i} &\dots\; H_k\to T_k
		}
		$$
		
		\item $(\div\to)$ if $e_0$ is not contained in any $H_i$ (then $e_0$ belongs to $E_K$):
		$$
		\infer[(\mathrm{cut})]{G[e_0/H]\to B}{
			H\to A
			&
			\infer[(\div\to)]{G\to B}{
				K\to B & H_1\to T_1\; & \dots & \; H_k\to T_k
			}
		}
		$$
		\begin{center}
			$\downsquigarrow$
		\end{center}
		$$
		\infer[(\div\to)]{G[e_0/H]\to B}{
			\infer[(\mathrm{cut})]{K[e_0/H]\to B}
			{H\to A & K\to B} &
			H_1\to T_1\; & \dots & \; H_k\to T_k
		}
		$$
		\item $(\times\to)$: 
		$$
		\infer[(\mathrm{cut})]{G[e_0/H]\to B}{
			H\to A
			&
			\infer[(\times\to)]{G\to B}{G[f_0/K]\to B}
		}
		$$
		Here $f_0\in E_G$ is labeled by $\times(K)$ (and $f_0\ne e_0$, because $A$ is not major). The remodelling is as follows:
		$$
		\infer[(\to\times)]{G[e_0/H]\to B}{
			\infer[(\mathrm{cut})]{G[f_0/K][e_0/H]\to B}{
				H\to A
				&
				G[f_0/K]\to B
			}
		}
		$$
		\item $(\to\div)$:
		$$
		\infer[(\mathrm{cut})]{G[e_0/H]\to B}
		{
			H\to A
			&
			\infer[(\to\div)]{G\to N\div D}{D[d_0/G]\to N}
		}
		$$
		Here $d_0\in E_D$ is labeled by \$. Then
		$$
		\infer[(\to\div)]{G[e_0/H]\to N\div D}
		{
			\infer[(\mathrm{cut})]{D[d_0/G[e_0/H]]\to N}{
				H\to A
				&
				D[d_0/G]\to N
			}
		}
		$$
		Here we use the associativity property: $D[d_0/G[e_0/H]]=D[d_0/G][e_0/H]$.
		\item $(\to\times)$: 
		$$
		\infer[(\mathrm{cut})]{G[e_0/H]\to \times(M)}{
			H\to A
			&
			\infer[(\to\times)]{G\to \times(M)}{
				H_1\to T_1\;\dots & H_i\to T_i &\dots\; H_k\to T_k 
			}
		}
		$$
		Here $G$ is composed of copies of $H_1,\dots,H_k$ by means of $M$. Since $e_0\in E_G$, there is such a graph $H_i$ that $e_0\in E_{H_i}$. Then we can remodel this derivation as follows:
		$$
		\infer[(\to\times)]{G[e_0/H]\to \times(M)}{
			H_1\to T_1\;\dots & \infer[(\mathrm{cut})]{H_i[e_0/H]\to T_i}
			{H\to A & H_i\to T_i} &\dots\; H_k\to T_k
		}
		$$
	\end{enumerate}
	
	\textbf{Case 5.} $A=\times(M)$ is major in both $H\to A$ and $G\to B$.
	$$
	\infer[(\mathrm{cut})]{G[e_0/H]\to B}
	{
		\infer[(\to\times)]{H\to\times(M)}{H_1\to T_1 & \dots & H_k\to T_k}
		&
		\infer[(\times\to)]{G\to B}{G[e_0/M]\to B}
	}
	$$
	Let us denote $E_M=\{e_1,\dots,e_k\}$ and $lab_M(e_i)=T_i$. Note that $M$ is a subgraph of $K\eqdef G[e_0/M]$, in particular $E_M\subseteq E_K$. Now we are ready to remodel this derivation as follows:
	$$
	\infer[(\mathrm{cut})]{K[e_1/H_1]\dots[e_k/H_k]\to B}
	{
		H_k\to T_k
		&
		\infer[]{K[e_1/H_1]\dots[e_{k-1}/H_{k-1}]\to B}
		{
			\infer[(\mathrm{cut})]{\dots}
			{
				H_2\to T_2
				&
				\infer[(\mathrm{cut})]{K[e_1/H_1][e_2/H_2]\to B}
				{
					\infer[(\mathrm{cut})]{K[e_1/H_1]\to B}
					{
						H_1\to T_1
						&
						K\to B
					}
				}
			}
		}
	}
	$$
	Finally, note that $K[e_1/H_1]\dots[e_k/H_k]=G[e_0/H]$. The induction hypothesis applied several times from top to bottom of this new derivation completes the proof.
	
	\textbf{Case 6.} $A=N\div D$ is major in both $H\to A$ and $G\to B$.
	$$
	\infer[(\mathrm{cut})]{G[e_0/H]\to B}
	{
		\infer[(\to\div)]{H\to N\div D}{D[d_0/H]\to N}
		&
		\infer[(\div\to)]{G\to B}{L\to B & H_1\to T_1 & \dots & H_k\to T_k}
	}
	$$
	Here $d_0\in E_D$ is labeled by \$. We denote edges in $E_D$ except for $d_0$ as $e_1,\dots,e_k$; let $lab_D(e_i)=T_i$ (from above). Note that $e_1,\dots,e_k$ can be considered as edges of $K\eqdef D[d_0/H]$ as well. Observe that $L$ has to contain an edge labeled by $N$ that participates in $(\div\to)$; denote this edge by $\widetilde{e}_0$. Then the following remodelling is done:
	$$
	\infer[(\mathrm{cut})]{L[\widetilde{e}_0/K][e_1/H_1]\dots[e_k/H_k]\to B}
	{
		H_1\to T_1
		&
		\infer[]{L[\widetilde{e}_0/K][e_1/H_1]\dots[e_{k-1}/H_{k-1}]\to B}
		{
			\infer[(\mathrm{cut})]{\dots}
			{
				H_2\to T_2
				&
				\infer[(\mathrm{cut})]{L[\widetilde{e}_0/K][e_1/H_1][e_2/H_2]\to B}
				{
					\infer[(\mathrm{cut})]{L[\widetilde{e}_0/K][e_1/H_1]\to B}
					{
						H_1\to T_1
						&
						\infer[(\mathrm{cut})]{L[\widetilde{e}_0/K]\to B}
						{
							K\to N
							&
							L\to B
						}
					}
				}
			}
		}
	}
	$$
	As a final note, we observe that $L[\widetilde{e}_0/K][e_1/H_1]\dots[e_k/H_k]=G[e_0/H]$. This completes the proof.
\end{proof}

\subsection{Theorem \ref{NP_der}}
\begin{proof}
	This problem is in NP: if $H\to A$ is derivable, then a certificate of derivability is a derivation tree of $H\to A$. This derivation tree has to include all steps of the derivation starting with axioms, and all isomorphisms between graphs in premises and in a conclusion that justify that a replacement (or a compression) is done correctly. Such a certificate has polynomial size w.r.t. size of $H\to A$ since the sum of sizes of all premises is strictly less than the size of a sequent in a conclusion (isomorphisms make it larger, but since each isomorphism can be represented as a list of correspondences between edges in graphs in premises and in a conclusion, their total size can be estimated by the size of a conclusion as well).
	
	NP-completeness directly follows from Theorem \ref{embed_lambek}: since the Lambek calculus is NP-complete, and it is embedded in $\mathrm{HL}$ (in polynomial time), the latter is NP-complete as well.
\end{proof}

\section{Embedding of $\mathrm{HL}$ in the intuitionistic logic}\label{sec_datalog_appendix}
In Section \ref{ssec_datalog} we only considered some examples of how to embed sequents of $\mathrm{HL}$ in $\FOInt$. Here we consider this embedding in general and prove its correctness.

Types of $\mathrm{HL}$ are converted into formulas as follows:
\begin{enumerate}
	\item A primitive type $p$ such that $type(p)=k$ is considered to be a $k$-ary predicate variable, so it can form expressions of the form $p(x_1,\dots,x_k)$. We denote this as $v(p)[x_1,\dots,x_k]\eqdef p(x_1,\dots,x_k)$.
	\item Let $N\div D$ be a type and let $E_D=\{d_0,\dots,d_l\}$ where $lab_D(d_0)=\$$. Let $type(d_0)=k$. We introduce $k$ variables $x_1,\dots,x_k$ and $m\eqdef |V_D|-k$ variables $y_1,\dots,y_m$. Let us assign each variable to a node of $D$ in such a way: we assign $x_1,\dots,x_k$ to nodes attached to $d_0$ and we assign $y_1,\dots,y_m$ to the rest of the nodes (in some order). Let $f$ be a function that takes a node and returns a variable assigned to it. Let $lab_D(d_i)=T_i$ for $i\ge 1$. Then we translate the type as follows:
	\begin{multline}\label{datalog_div}
			v(N\div D)[x_1,\dots,x_k]\eqdef\\ \forall y_1\dots \forall y_m \big(v(T_1)[f(att_D(d_1))]\;\&\;\dots\;\&\;v(T_l)[f(att_D(d_l))]\to v(N)[f(ext_D)]\big)
	\end{multline}
	\item Let $\times(M)$ be a type and let $E_M=\{m_1,\dots,m_l\}$; let also $lab_M(m_i)=T_i$. For $k\eqdef |ext_M|$ we introduce $k$ new variables $x_1,\dots,x_k$ along with $m\eqdef |V_M|-k$ new variables $y_1,\dots,y_m$. Again, there is a function $g$ that bijectively assigns one of variables $x_1,\dots,x_k$ to external nodes of $M$ and one of variables $y_1,\dots,y_m$ to the rest of the nodes. Then the traslation is the following:
	\begin{multline}\label{datalog_times}
	v(\times(M))[x_1,\dots,x_k]\eqdef\\ \exists y_1\dots \exists y_m\big(v(T_1)[g(att_M(m_1))]\;\&\;\dots\;\&\; v(T_l)[g(att_M(m_l))]\big)
	\end{multline}
	\item A sequent $H\to A$ is translated as follows (where $k=type(H)$): 
	\begin{multline}
	v(H\to A)[x_1,\dots,x_k]\eqdef v(\times(H))[x_1,\dots,x_k]\to v(A)[x_1,\dots,x_k]
	\end{multline}
\end{enumerate}
Free variables in these formulas may be considered as universally quantified. An important question is what happens when $l=0$. In such a case there are 0 conjuncts in a formula, so we can write $\top$ instead.

Such formulas correspond to a particular case of embedded implications in the sense of \cite{Bonner89}. E.g., look at the following query taken from \cite{Bonner89}:
$$
easy(d) \leftarrow \forall s\big(grad(s,d)\leftarrow take(s,his100),take(s,eng100)\big)
$$
This formula defines that a department in a university is \emph{easy} if any student can graduate by taking courses ``history 100'' and ``english 100''. This formula differs from formulas occuring above in translations since there are constants $his100$ and $eng100$; in order to avoid this problem let us say that $h_{100}$ ($e_{100}$) is a unary predicate that takes an item as an input and says whether it is a course of history 100 (english 100 resp.) or not. I.e. the above formula is converted into the one
$$
easy(d) \leftarrow \forall s\big(grad(s,d)\leftarrow \exists c_1\exists c_2[take(s,c_1)\;\&\;take(s,c_2)\;\&\; h_{100}(c_1)\;\&\; e_{100}(c_2)]\big)
$$
Equivalently,
$$
easy(d) \leftarrow \forall s\forall c_1\forall c_2\big(grad(s,d)\leftarrow [take(s,c_1)\;\&\;take(s,c_2)\;\&\; h_{100}(c_1)\;\&\; e_{100}(c_2)]\big)
$$

Translating this formula back to $\mathrm{HL}$ we can say that the property of being easy is defined by the following type:
$$
grad\div
\left(\mbox{{\tikz[baseline=.1ex]{
			\node (V) {};
			\node[node, below=0.7mm of V, label=left:{\scriptsize $(2)$}] (N1) {};
			\node[hyperedge,above=3mm of N1] (E1) {$\$$};
			\node[node,above right=3mm and 12mm of N1, label=left:{\scriptsize $(1)$}] (N2) {};
			\node[node,above right=3mm and 10mm of N2] (N3) {};
			\node[node,below right=3mm and 10mm of N2] (N4) {};
			\node[hyperedge,right=6mm of N3] (E2) {$\;h_{100}\;$};
			\node[hyperedge,right=6mm of N4] (E3) {$\;e_{100}\;$};

			\draw[-,black] (N1) -- node[right] {\scriptsize 1} (E1);
			\draw[->,black] (N2) -- node[above] {$take$} (N3);
			\draw[->,black] (N2) -- node[below] {$take$} (N4);
			\draw[-,black] (N3) -- node[above] {\scriptsize 1} (E2);
			\draw[-,black] (N4) -- node[below] {\scriptsize 1} (E3);
}}}\right)
$$
This is the way we can look at types of $\mathrm{HL}$ as at formulas of the intuitionistic logic or of Datalog with embedded implications.
\begin{proof}[Proof sketch of Theorem \ref{fo-int}.]
	Induction on length of a derivation of $H\to A$.
	
	The axiom $p\to p$ is translated into the true formula $p(\overline{x})\to p(\overline{x})$. Now we consider variants depending on the last rule (notation is taken from definitions in Section \ref{sec_axioms_rules}):
	
	Case $(\div\to)$:
	$$
	\infer[(\div\to)]{H[e/D][d_0\eqdef N\div D][d_1/H_1,\dots,d_k/H_k]\to A}{H\to A & H_1\to lab(d_1) &\dots & H_k\to lab(d_k)}
	$$
	By the induction hypothesis, all the premises are converted into intuitionitically true formulas. Note that $v(H\to A)[\overline{x}]=\exists\overline{y}(\dots\&\;v(N)[z_1,\dots,z_m]\;\&\dots)\to v(A)[\overline{x}]$. After the replacement of $e$ by $D$ and the relabeling we obtain the formula
	$$
	\exists\overline{y}(\dots\&\;v(\times(D[d_0\eqdef N\div D]))[z_1,\dots,z_m]\;\&\dots)\to v(A)[\overline{x}]
	$$
	The main point is that $\vdash v(\times(D[d_0\eqdef N\div D]))[z_1,\dots,z_m]\to v(N)[z_1,\dots,z_m]$; this directly follows from the definition of $v$. Let us show this on example:
	\begin{example}
		Consider the sequent $(p,s\div (p\$ q)^\bullet, q)^\bullet\to s$. It is translated into the formula $\exists y\exists z\big(p(x,y)\;\&\;\forall a\forall b(p(a,y)\;\&\;q(z,b)\to s(a,b))\;\&\;q(z,t)\big)\to s(x,t)$. It is obviously true in all Kripke models. 
	\end{example}
	$v(\times(D[d_0\eqdef N\div D]))[z_1,\dots,z_m]$ is a conjunction of several formulas under the existential quantifier; those formulas are $v(lab(d_i))$ of some variables. However, we know that $\vdash v(\times(H_i))\to v(lab(d_i))$; therefore, we can change each conjunct of the form $v(lab(d_i))$ with the formula $v(\times(H_i))$ preserving truthiness. Finally, we obtain the translation of the conclusion.
	
	Case $(\to\div)$:
	$$
	\infer[(\to\div)]{F\to N\div D}{D[e_0/F]\to N}
	$$
	As in the definition of $v$, let us assign variables to nodes. Let the premise be translated into a formula of the form
	$$
	\exists \overline{y}(v(T_1)[\overline{x_1}]\;\&\;\dots\;\&\; v(T_k)[\overline{x_l}])\to v(N)[\overline{x}]
	$$
	Let us say without loss of generality that the first $m$ conjuncts $v(T_i)[\overline{x_i}]$, $i=1,\dots,m$ correspond to edges of $D[e_0/F]$ from $F$, and the rest of them correspond to other edges. Then the conclusion is transformed into the formula of the form
	$$
	\exists\overline{y^\prime}(v(T_1)[\overline{x_1}]\;\&\;\dots\;\&\; v(T_m)[\overline{x_m}])\to \forall \overline{z}\big(v(T_{m+1})[\overline{x_{m+1}}]\;\&\;\dots\;\&\;v(T_l)[\overline{x_l}]\to v(N)[\overline{x}]\big)
	$$
	Variables in $\overline{y^\prime}$ are included in those in $\overline{y}$: all nonexternal nodes of $F$ are also nonexternal in $D[e_0/F]$. Informally, the first formula says that existence of elements participating in relations $v(T_1),\dots,V(T_l)$ implies $v(N)$; the second formula says that if there are elements participating in a smaller number of relations $v(T_1),\dots,v(T_m)$, then, after their arbitrary supplementing with elements corresponding to variables $\overline{z}$ in such a way that the rest of relations hold, $v(N)$ holds on appropriate elements. Obviously, the latter semantically follows from the former.
	\begin{example}
		The derivable sequent $(pq)^\bullet\to \times((pqs)^\bullet)\div(\$ s)^\bullet$ corresponds to the formula
		$$
		\exists y\big(p(x,y)\;\&\;q(y,z)\big)\to\forall t\big(s(z,t)\to \exists a\exists b[p(x,a)\;\&\;q(a,b)\;\&\;s(b,t)]\big)
		$$
	\end{example}
	
	Case $(\times\to)$: it is just moving existential quantifiers inside conjunctions, when this is allowed. Look at the following example:
	\begin{example}
		A sequent $(p,q,r,s)^\bullet\to T$ is translated into the formula
		$$
		\exists y\exists z\exists u\big(p(x,y)\;\&\; q(y,z)\;\&\; r(z,u)\;\&\;s(u,v)\big)\to v(T)(x,v).
		$$
		In comparison, the sequent 		$(p,\times((q,r)^\bullet),s)^\bullet\to T$ is converted as follows:
		$$
		\exists y\exists u\big(p(x,y)\;\&\; \exists z(q(y,z)\;\&\; r(z,u))\;\&\;s(u,v)\big)\to v(T)(x,v)
		$$
	\end{example}
	
	Case $(\to\times)$. At the formula level this rule is the following:
	$$
	\infer[]{
		\exists \overline{y} \big(v(\times(H_1))[\overline{y_1}]\;\&\;\dots\;\&\;v(\times(H_l))[\overline{y_l}]\big)\to \exists \overline{y}(v(lab(m_1))[\overline{y_1}]\;\&\;\dots\;\&\;v(lab(m_l))[\overline{y_l}])
	}
	{
		v(\times(H_1))[\overline{x_1}]\to v(lab(m_1))[\overline{x_1}] & \dots & v(\times(H_l))[\overline{x_l}]\to v(lab(m_l))[\overline{x_l}]
	}
	$$
	Clearly, the lower formula follows from the upper one.
\end{proof}
The opposite statement does not, however, hold: if $\FOInt\vdash v(H\to A)$, then this does not imply that $\mathrm{HL}\vdash H\to A$. Moreover, in some cases two different formulas (one of which is derivable and the other one is not) take the same interpretation w.r.t. $v$. A series of such ``bad'' examples is presented below:
\begin{center}
	\begin{tabular}{|c|c|}
		\hline
		Sequent & Formula \\
		\hline
		$\mbox{{\tikz[baseline=.1ex]{
		\node (V) {};
		\node[hyperedge,above=-3mm of V] (E1) {$p$};
		\node[node,left=4mm of E1, label=left:{\scriptsize $(1)$}] (N1) {};
		\node[hyperedge,right=9mm of E1] (E2) {$p$};
		\node[node,left=4mm of E2] (N2) {};
		\node[above = 2mm of N1] {};
		\node[below = 2mm of N1] {};
		
		\draw[] (N1) -- (E1);
		\draw[] (E2) -- (N2);
		}}}\to p$
		 & $\exists y(p(x)\;\&\;p(y))\to p(x)$ \\
		\hline
		$p\to tr_{\mathrm{P}}(p\cdot p)$ & 
		$p(x)\to (p(x)\;\&\;p(x))$
		\\
		\hline
		\begin{tabular}{c}
		$\langle\{v_0,v_1\},\emptyset,\emptyset,\emptyset,v_0v_1\rangle\to \times(\langle\{w_0\},\emptyset,\emptyset,\emptyset,w_0\rangle)$ \\
		$\langle\{v_0,v_1\},\emptyset,\emptyset,\emptyset,v_0v_1\rangle\to \times(\langle\{w_0,w_1\},\emptyset,\emptyset,\emptyset,w_0w_1\rangle)$
		\\
		\end{tabular}
		&
		$\top\to \top$
		\\
		\hline
	\end{tabular}
\end{center}
Possible reasons for incompleteness of such an embedding are discussed in \cite{Pshenitsyn20Preprint}.
\end{document}